\pdfoutput=1

\documentclass[]{article}
\usepackage[margin=1in]{geometry}
\usepackage{amsmath}
\usepackage{amsfonts}
\usepackage{amssymb}
\usepackage{amsthm}
\usepackage{esint}
\usepackage{mathtools}
\usepackage{array,makecell,longtable}
\usepackage{titling}
\usepackage{authblk}
\usepackage{commath}
\usepackage{xcolor}
\usepackage{graphicx}
\usepackage[toc,title,page]{appendix}
\usepackage{subcaption}
\usepackage{bbm}
\usepackage{cite,hyperref}
\usepackage{color,soul}
\usepackage{bm}
\usepackage{mathrsfs}

\usepackage{tikz}
\usepackage{enumitem}
\newtheorem{remark}{Remark}[section]
\newtheorem{theorem}[remark]{Theorem}
\newtheorem{corollary}[remark]{Corollary}
\newtheorem{lemma}[remark]{Lemma}

\newtheorem{definition}[remark]{Definition}

\numberwithin{equation}{section}
\allowdisplaybreaks

\title{Inverse problems for coupled nonlocal nonlinear systems arising in mathematical biology}

\author{Ming-Hui Ding\thanks{School of Mathematics and Statistics, Northwestern Polytechnical University, Xi'an, Shaanxi Province, China\\ Email address: dingmh@nwpu.edu.cn}, \, Hongyu Liu\thanks{Department of Mathematics, City University of Hong Kong, Hong Kong SAR, China\\ Email address: hongyu.liuip@gmail.com, hongyliu@cityu.edu.hk} \, and Catharine W.K. Lo\thanks{Department of Mathematics, City University of Hong Kong, Hong Kong SAR, China\\ Email address: wingkclo@cityu.edu.hk} \vspace{-0.5cm}}

\date{}

\begin{document}

\maketitle

\begin{abstract}
    In this paper, we propose and study several inverse problems of determining unknown parameters in nonlocal nonlinear coupled PDE systems, including the potentials, nonlinear interaction functions and time-fractional orders. In these coupled systems, we enforce non-negativity of the solutions, aligning with realistic scenarios in biology and ecology. There are several salient features of our inverse problem study: the drastic reduction in measurement/observation data due to averaging effects, the nonlinear coupling between multiple equations, and the nonlocality arising from fractional-type derivatives. These factors present significant challenges to our inverse problem, and such inverse problems have never been explored in previous literature. To address these challenges, we develop new and effective schemes. Our approach involves properly controlling the injection of different source terms to obtain multiple sets of mean flux data. This allows us to achieve unique identifiability results and accurately determine the unknown parameters. Finally, we establish a connection between our study and practical applications in biology, further highlighting the relevance of our work in real-world contexts.


    
\end{abstract}

\section{Introduction}

\subsection{Mathematical Setup}
Focusing first mainly on the mathematics, but not the physical or biological applications, we begin by introducing the mathematical setup of the study.  
Let $\Omega$ be a bounded domain in $\mathbb{R}^d$ $(d\geq1)$. For $u(x):\Omega\rightarrow \mathbb{R}$, define the action of the nonlocal diffusion operator $\mathcal{L}$ on the function $u(x)$ as follows
\begin{equation}
\label{NonlocalLDef}
\mathcal{L}u(x):=2\int_{\mathbb{R}^d}(u(y)-u(x))\gamma (x,y)\,dy,\quad \forall x\in \Omega\subseteq \mathbb{R}^d.
\end{equation}
Here, the kernel $\gamma(x,y):\mathbb{R}^d\times\mathbb{R}^d\rightarrow \mathbb{R}$ is a non-negative symmetric function, and $\gamma=\sum_{i=1}^N\gamma_i$ satisfies the following inequalities
\begin{equation}
\label{KernelAssump1}
\gamma^*\geq\gamma_i(x,y)| x-y|^{n+2s_i}\geq\gamma_{*},\quad \text{for }y\in \mathbb{R}^d,
\end{equation}
where $0<s_1<\cdots<s_N<1$, $\gamma^*$ and $\gamma_*$ are positive constants uniform for all $i$. Hence, the operator $\mathcal{L}$ is nonlocal in the sense that its value at $x$ depends on information about $u$ at all points $y\neq x$.

Consider the following coupled nonlinear nonlocal parabolic system with a Dirichlet boundary condition:
\begin{equation}\label{SpaceFracMainPb}
\begin{cases}
\partial_t \textbf{u}-\mathcal{L} \textbf{u}+\bm{\alpha}\cdot\nabla \textbf{u}=\textbf{p}(x)\textbf{u}+\textbf{F}(x,t,\textbf{u})+\textbf{q} &\text{ in }\  \Omega\times(0,T],\\
\textbf{u}=0&\text{ on }\  \Omega^c\times(0, T],\\
\textbf{u}(x,0)=0,&\text{ in }\ \mathbb{R}^d,\\
\textbf{u} \geq \mathbf{0} &\text{ in }\ \bar{\Omega}\times[0,T],
\end{cases}
\end{equation}
where the convection coefficients $\bm{\alpha}=(\alpha_1,\dots,\alpha_M)$ are vectors in $\mathbb{R}^d$ with constant entries, and the reaction potentials $\textbf{p}=(p_1,\dots,p_M)$, $p_i(x)\in C^{\gamma,\gamma/2}(\bar{\Omega})$ are negative functions. This model describes anomalous spatial diffusion \cite{DuGunzburgerLehoucqZhouNonlocalDiffusion,Silling2000NonlocalLevyJumpPeridynamics,DuJuLu2019Galerkin1DNonlocalDiffusion,TianDu2013NonlocalDiffusion}, and some biological applications will also be discussed in Section \ref{sec:apply}. Here, the nonlocal Dirichlet boundary condition implies that the region outside $\Omega$ is considered a hostile environment biologically, and any individual that jumps outside of $\Omega$ dies instantly \cite{TangDai2024NonlocalAdvectionPopulation}. Therefore, no organism can leave or enter the region $\Omega$.  

For this model, we are interested in studying the potential functions $\mathbf{p}(x)$ and nonlinear interaction functions $\mathbf{F}$, by knowledge of the nonlocal flux data of $\mathbf{u}$ on an accessible region $\Omega_a\subseteq\mathbb{R}^d\backslash\Omega$ in the exterior of $\Omega$ (see Figure \ref{fig:space} for a schematic illustration). Formally, this is given by the measurement map \begin{equation}\label{FormalMapExt}\Lambda^1_{\mathbf{p},\mathbf{F}}:\mathbf{q}\to \mathbf{p},\mathbf{F}.\end{equation} 

\begin{figure}
\begin{center}
\begin{tikzpicture}[line width=1.5pt,scale=0.7]
\draw[very thick] (0,0) to [out=110,in=185] (2,2) to [out=15,in=140] (4,2) to [out=-30,in=50] (5,-1) to [out=-150,in=-50] (0,0);
\draw[very thick] (7,0) to [out=90,in=190] (8,1) to [out=190,in=140] (9,1) to [out=-30,in=50] (10,-1) to [out=-150,in=-80] (7,0);
\node at (4,1.5){$\mathbf{\Omega}$};
\node at (6.3,2){$\mathbb{R}^d\setminus\mathbf{(\Omega\cup\Omega_a)}$};
\node at (7.7,0.2){$\mathbf{\Omega_{a}}$};
\node at (1.5,0.5){$\bm{\otimes}$};
\node at (2,-0.5){$\bm{\otimes}$};
\node at (4,0.5){$\bm{\otimes}$};
\node at (2.5,1){$\bm{\otimes}$};
\node at (3,0.09){$\bm{\otimes}$};
\node at (3.5,-0.6){$\bm{\otimes}$};
\node at (8.2,-0.5){$\bm{\odot}$};
\node at (9.3,0){$\bm{\odot}$};
\node at (8.6,0.5){$\bm{\odot}$};
\node at (9,-0.8){$\bm{\odot}$};
\end{tikzpicture}
\end{center}
\caption{The schematic illustration of the coupled nonlinear nonlocal parabolic inverse problem: $\Omega$ is the physical domain, $\Omega_{a}$ denotes the accessible region, $\bm{\otimes}$ represents the input source location, and $\bm{\odot}$ indicates the measurement location.}
\label{fig:space}
\end{figure}
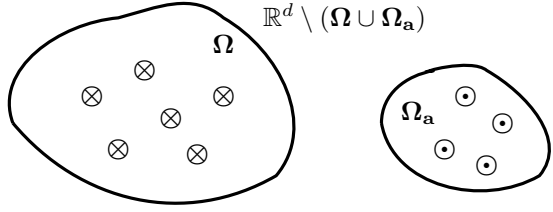

We consider the measurement given by $\langle\Lambda^1_{\mathbf{p},\mathbf{F}},\mathbf{h}\rangle$, $\mathbf{h}=(h_1,\dots,h_M)$, for some smooth function $h_i$ ($i=1,\dots,M$) which characterises the measure instrument. This means that our measurement map represents weighted integral data from an accessible part of the region. This is more often the case in physical applications \cite{DingZheng2021IPSpaceFrac}. However, due to the averaging effect, it can only provide very limited information, and also leads to more serious ill-posedness for the inverse problem. More details about this measurement map will be provided in Section \ref{subsect:MainResDef}.

We are mainly concerned with the unique identifiability issue of the inverse problem \eqref{FormalMapExt}. In a formal manner, our main result for the coupled nonlinear nonlocal parabolic inverse problem can be roughly summarised into the following theorem:
\begin{theorem}\label{FormalThmSpace}
    Suppose that $\mathbf{p},\mathbf{F}$ belong to (different) general a-priori function spaces. Given the inject source $\mathbf{q}$, let $\Lambda^1_{\mathbf{p}^k,\mathbf{F}^k}$ be the measurement map associated to \eqref{SpaceFracMainPb} for $k=1,2$. If 
    \[\Lambda^1_{\mathbf{p}^1,\mathbf{F}^1}=\Lambda^1_{\mathbf{p}^2,\mathbf{F}^2}\] for all $\mathbf{q}$, then it holds that 
    \[\mathbf{p}^1=\mathbf{p}^2\quad \text{ and }\quad \mathbf{F}^1=\mathbf{F}^2\quad \text{ in }\Omega\times(0,T).\]
\end{theorem}

We will establish the conditions under which Theorem \ref{FormalThmSpace} holds in Section \ref{subsect:MainResDef}, in order to ensure that the nonlocal exterior flux data can uniquely determine the reaction functions $\mathbf{p}$ and interaction functions $\mathbf{F}$. 

Another form of nonlocality can be considered when the diffusion phenomena occurs in highly heterogeneous media with memory and hereditary properties \cite{Metzler20001,Kilbas,SamkoKilbas,UchaikinI,UchaikinII}. This type of anomalous diffusion is governed by the following time-fractional diffusion equation: 
\begin{equation}\label{TimeMainPb0}
\begin{cases}
{_{0}D}_t^{\bm{\beta}}\mathbf{u}-\mathbf{d}\Delta \mathbf{u}+\bm{\alpha}\cdot\nabla \mathbf{u}=\mathbf{p}(x,t)\mathbf{u}+\mathbf{F}(x,t,\mathbf{u})+\mathbf{q} &\text{ in }\  \Omega\times(0,T],\\
 \mathbf{u}=0&\text{ on }\  \partial\Omega\times(0, T),\\
\mathbf{u}(x,0)=0&\text{ in }\  \Omega,\\
\mathbf{u}\geq \mathbf{0}&\text{ in }\  {\bar{\Omega}\times[0,T]},
\end{cases}
\end{equation}
where the fractional time derivative ${_{0}D}_t^{\bm{\beta}}$ represents
\begin{equation}
    ({_{0}D}_t^{\bm{\beta}}\mathbf{u})_i:=\sum_{j=1}^{B_i}b_{j.i}({_{0}D}_t^{\beta_{j,i}}u_i)\quad i=1,\dots,M,
\end{equation}
for fixed positive integers $B_i$ and positive constants $b_{j,i}$. The fractional orders satisfy $0<\beta_{1,i}<\beta_{2,i}<\cdots<\beta_{B_i,i}<1$, and for a function $v$, $_{0}D_t^{\alpha}v$ is the Caputo fractional derivative defined by \cite{Caputo}
\begin{equation}\label{CaputoDer}_{0}D_t^{\alpha}v=\frac{1}{\Gamma(1-\alpha)}\int_0^t(t-s)^{-\alpha}\frac{\partial v(x,s)}{\partial s}ds,\end{equation}
and $\Gamma(\cdot)$ denotes the Gamma function.
Also, the functions $\mathbf{p}$ in \eqref{TimeMainPb0} are such that $\mathbf{p}=(p_1,\dots,p_M)$, $p_i(x,t)\in C^{\gamma,\gamma/2}(\bar{\Omega}\times[0,T])$ are negative functions.

Unlike the nonlocal parabolic system, we use the local Dirichlet boundary condition here, such that the condition $\mathbf{u}=0$ is enforced only on the boundary $\partial\Omega$. From a biological perspective, this means that there are no organisms on the boundary of $\partial\Omega$, but there may be some in the exterior $\Omega^c$ away from the boundary.

For this model, we are interested in studying the potential functions $\mathbf{p}(x,t)$ and nonlinear interaction functions $\mathbf{F}$, as well as the fractional orders $\beta_{j,i}$. The first two are determined by the knowledge of the flux data of $\mathbf{u}$ on a portion $\Gamma\subset\partial\Omega$ of the boundary $\partial\Omega$ (see Figure \ref{fig:time} for a schematic illustration), while the last one is obtained from the observation of the values of $\mathbf{u}$ at an interior point $x_0\in\Omega$ for all times $t\in(0,T)$. Formally, the measurement maps can be expressed in the following general form \begin{equation}\label{FormalMapTime}\Lambda^2_{\mathbf{p},\mathbf{F}}:\mathbf{q}\to \mathbf{p},\mathbf{F},\quad \Lambda^3_{\bm{\beta}}:\mathbf{q}\to \bm{\beta}.\end{equation} 

\begin{figure}
\centering
  \includegraphics[width=2.5in]{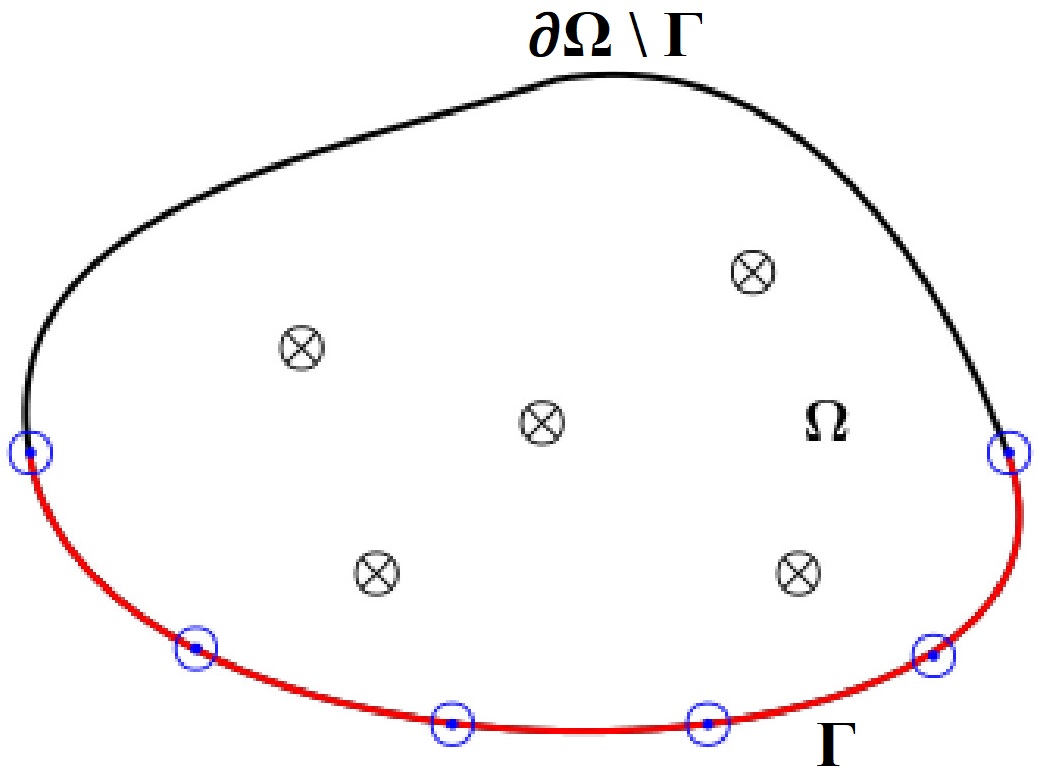}
  \caption{ The physical domain: $\Omega$,\ \ accessible boundary: $\Gamma$,\ \
  inaccessible boundary: $\partial\Omega\setminus\Gamma$,\ \ input source locations: $\otimes$,\ \ measurement locations: $\odot$. }
  \label{fig:time}
\end{figure}

Once again, we only consider the measurement given by $\langle\Lambda^2_{\mathbf{p},\mathbf{F}},\mathbf{h}\rangle$, for some smooth measuring means $h_i$. This means that our measurement map is only a weighted integral data of the partial boundary, as is often the case in physical situations \cite{DingZheng2020IPTimeFracHydrology}. Despite the decreased amount of measurement data and increased ill-posedness of this inverse problem, we are able to show that this limited average flux measurement is enough to uniquely identify the potentials $\mathbf{p}$ and interaction functions $\mathbf{F}$.

We are mainly concerned with the unique identifiability issue of the inverse problem \eqref{FormalMapTime}. In a formal manner, our main result for the time-fractional diffusion inverse problem can be roughly summarised into the following theorem:
\begin{theorem}\label{FormalThmTime}
    Suppose that $\mathbf{p},\mathbf{F}$ belong to general a-priori function spaces. Given the inject source $\mathbf{q}$, let $\Lambda^2_{\mathbf{p}^k,\mathbf{F}^k},\Lambda^3_{\bm{\beta}^k}$ be the measurement maps associated to \eqref{TimeMainPb0} for $k=1,2$. If 
    \[\Lambda^2_{\mathbf{p}^1,\mathbf{F}^1}=\Lambda^2_{\mathbf{p}^2,\mathbf{F}^2}\quad \text{ and }\quad \Lambda^3_{\bm{\beta}^1}=\Lambda^3_{\bm{\beta}^2}\] for all $\mathbf{q}$, then it holds that 
    \[\mathbf{p}^1=\mathbf{p}^2,\quad \bm{\beta}^1=\bm{\beta}^2 \text{ and }\quad \mathbf{F}^1=\mathbf{F}^2\quad \text{ in }\Omega\times(0,T).\]
\end{theorem}

We will establish the conditions for Theorem \ref{FormalThmTime} to hold in Section \ref{subsect:MainResDef}. 
Observe that in this case, $\mathbf{p}$ depends on both the spatial variable $x$ and time variable $t$, as opposed to the case of the nonlocal parabolic system. Correspondingly, we expect that a larger amount of average flux data is required to reconstruct $\mathbf{p}(x,t)$ in comparison to $\mathbf{p}(x)$ (see Remark \ref{MoreDataPxtRemark}).

\subsection{Background Motivation and Technical Developments}

Coupled systems of partial differential equations have been extensively studied and applied to model various interconnected physical, biological, and social phenomena \cite{May1976CoupledBioModel,Du2002PopulationCoupled,PengWangZhangZhou2023CoupledEpidemic,JinWangWu2023AlarmTaxisCoupled,LouYuan2020JMBPopulationCoupled,LouYuan2020ProcAMSPopulationCoupled,gaston2009coupled,Ambrosetti2007coupled,wang2004degenerate}. These systems often involve complex dynamics and interactions between different components. However, in many real-world scenarios, the dynamics of the interconnected components exhibit long memory effects and non-local behavior, which cannot be adequately captured by classical differential equations. Fractional calculus provides a powerful mathematical framework to address these challenges.

Henceforth, in recent years, the study of fractional single equations and coupled systems has gained significant attention across multiple scientific disciplines. In particular, the associated forward problems have been extensively studied  \cite{DuJuLu2019Galerkin1DNonlocalDiffusion,TianDu2013NonlocalDiffusion,NochettoOtarolaSalgado2015FracDiffusion,DuGunzburgerLehoucqZhouNonlocalDiffusion,SakamotoYamamoto2011JMAATimeFracWave,Luchko2011MultiTimeFracMaxPrinciple,Luchko2010TimeFracEqIBVPExist,LuchkoYamamoto2016TimeFracExistence,GorenfloLuchkoYamamoto2015FCAATimeFracEqExist,JinLazarovPasciakZhou2015TimeFracFEM}. However, results are still limited for the corresponding inverse problems. There have been some recent works on the inverse problem for the single fractional equation. 
Some works for the parabolic space-fractional equation include \cite{DEliaGunzburger2016IdentifyDiffusionNonlocal,JiaWu2018CarlemanEstimateFracDiffusion,JingJiaPeng2020NonlocalDiffusionInversePb,DingZheng2021IPSpaceFrac}, 
while works for the time-fractional diffusion equation include inverse initial boundary value problems \cite{JinRundell2015IP_InverseTimeFrac,LiuYamamoto2010InverseTimeFrac,Murio2008InverseHeatCaputo,LiuYamamotoYan2015InverseTimeFrac,RundellXuZuo2013InverseTimeFrac,ZhangYangLi-CAM2024-IdentifySourceInitialHadamardFrac-Numerical,JingJiaSong2024AMLTimeFracInversePb}, 
inverse source problems \cite{WeiLiLi2016IP_InverseTimeFracSource,LiuDuLi2024TimeFracInverse,ZhangXu2011IP-InverseTimeFrac,Liu2017MultiTermTimeFracInverseSource,LiuRundellYamamotoFCAAInverseTimeFracSource,ZhangYangLi-CAM2024-IdentifySourceInitialHadamardFrac-Numerical,SakamotoYamamoto2011JMAAFracTimeInversePb,JinRundell2015IP_InverseTimeFrac}, 
inverse coefficient problems \cite{MillerYamamoto2013InverseTimeFracOrderLinear,YanWei2023-InverseProbFracOrderLinear,JinRundell2012IP_InverseTimeFrac1D,JingJiaSong2024AMLTimeFracInversePb,JinRundell2015IP_InverseTimeFrac,DingZheng2020IPTimeFracHydrology,KaltenbacherRundell2019FracTimeInversePbPotential,JinZhou2021TimeFracInversePbOrderPotential,JinRundell2012IP_InverseTimeFrac1D}, 
and inverse fractional order problems \cite{LiZhangJiaYamamoto2013IP_InverseDiffusionCoefTimeFracOrder,MillerYamamoto2013InverseTimeFracOrderLinear,RundellYamamoto2023InverseTimeFracOrder,YanWei2023-InverseProbFracOrderLinear,JinKian2022SIMATimeFracOrderInversePb,JinKian2023FracTimeOrderInversePb,JinKian2021FracTimeOrderInversePb,JinZhou2021TimeFracInversePbOrderPotential,LiuYamamoto2023IPTimeFracCoupledEquationInversePb,LiLiuYamamoto2019IPTimeFracOrder,HongJinKian2024identificationspatiallydependentvariableorder}. 
Space-time fractional problems have also been considered in the theoretical setting in \cite{lin2023calderonproblemnonlocalparabolic,LaiLinRuland2020SIMASpaceTimeTogetherFracInversePb}, and with numerical methods in \cite{SongZhengJiang2021-SpaceTimeNonlocalBayesNumericsOnly,Tartar+T+Ulusoy2016SpaceTimeFracNumerics}. 
At the same time, inverse problems have been considered for coupled parabolic systems in the local classical case (see, for instance \cite{LiuMouZhang2022InversePbMeanFieldGames,LiuZhang2022-InversePbMFG,klibanov2023lipschitz,klibanov2023mean1,klibanov2023mean2,LiuZhangMFG3,klibanov2023holder,liu2023stability,imanuvilov2023lipschitz1,imanuvilov2023unique,klibanov2023coefficient1,klibanov2023coefficient2,imanuvilov2023global,imanuvilov2023lipschitz2,ding2023determining,LiuZhangMFG4,liu2023determining,li2023inverse,LiLoCAC2024,li2024inverse,liulo2024determiningstatespaceanomalies,DouYamamoto2019InversePbCoupledSchrodinger,LinLiuLiuZhang2021-InversePbSemilinearParabolic-CGOSolnsSuccessiveLinearisation,CGNP2013InversePbCoupledParabolic,RoquesCristofol2012InversePbPredPrey,LiuLoZhang2024decodingMFG,DingLiuZheng2023JMBInversePbBio}).

Yet, there are hardly any works on inverse problems for fractional coupled systems. The only existing works in this direction are \cite{RenHuangYamamoto2021JIIPTimeFracInversePbCoupled,LiHuangLiu2024-TimeFracCoupledDiffusion-InversePbIdentifyExponent,SratiOulmelkAfraitesHadri2023DCDSSInversePbTimeFracSystem,FengLiuLu2024CoupledDiffusionInversePbNumerics} and they only addressed the recovery of the fractional orders, potentials or initial values for the time-fractional coupled diffusion system. Unlike previous works, our inverse problem is different. Our main novelties lie in the following key aspects:
\begin{itemize}
    \item Firstly, we consider the inverse problem for the space-nonlocal nonlinear coupled parabolic system, and recover the potentials $\mathbf{p}$ and nonlinear interaction functions $\mathbf{F}$, using weighted measurements on an accessible exterior region $\Omega_a\subset\Omega^c$. Such a problem has never been considered elsewhere, to the best of our knowledge.
    \item Secondly, we consider the inverse problem for the time-fractional diffusion system, and recover the potentials $\mathbf{p}$, nonlinear interaction functions $\mathbf{F}$ and fractional orders $\bm{\beta}$. This is the first theoretical work considering such inverse problems.
    \item Thirdly, we give some applications of our results to biological situations.
\end{itemize}

Indeed, there have not been any previous studies on the inverse problem for the space-nonlocal nonlinear coupled parabolic system. Previous works have only considered the nonlocal parabolic problem with a single equation, as in \cite{DEliaGunzburger2016IdentifyDiffusionNonlocal,JiaWu2018CarlemanEstimateFracDiffusion,JingJiaPeng2020NonlocalDiffusionInversePb,DingZheng2021IPSpaceFrac}. In this work, we consider a nonlocal parabolic system, with an additional local drift. Such a mixed local-nonlocal equation arises when the phenomenon under investigation undergoes anomalous diffusion \cite{meerschaert2006fractional,DuGunzburgerLehoucqZhouNonlocalDiffusion} together with advective transport \cite{TangDai2024NonlocalAdvectionPopulation} (e.g. shifting habitats, river organisms being washed away by currents or viruses being carried by the wind). However, the inverse problems for this type of system have never been studied in existing literature. 

We address this inverse problem, using the weighted measurement map \eqref{FormalMapExt}. This leads to less data as we explained in the previous section, thereby increasing the difficulty of this problem. Nevertheless, we are still able to obtain the unique identifiability results for $\mathbf{p}$ and $\mathbf{F}$, as shown in the subsequent sections. Furthermore, the nonlinearity and coupling in the interaction functions $\mathbf{F}$ pose major difficulties to our analysis. Previous works in the local case have relied on the construction of complex-geometric-optics (CGO) solutions \cite{LiuMouZhang2022InversePbMeanFieldGames,LiuZhang2022-InversePbMFG,LiuZhangMFG3,LiuZhangMFG4,liu2023determining,li2023inverse,LiLoCAC2024,li2024inverse,liulo2024determiningstatespaceanomalies,LinLiuLiuZhang2021-InversePbSemilinearParabolic-CGOSolnsSuccessiveLinearisation,LiuLoZhang2024decodingMFG}, which do not make sense in the nonlocal case. Therefore, we devise a new method, by incorporating ideas from the case of the single nonlocal equation in \cite{DingZheng2021IPSpaceFrac} and the case of the coupled local system in  \cite{DingLiuZheng2023JMBInversePbBio}, and rely on carefully chosen input source functions $\mathbf{q}$ and high-order linearisation to obtain our result.

Even in the case of the time-fractional diffusion system, such weighted integral data on an accessible part of the
boundary pose great difficulty. Compared with the direct flux measurement (or Neumann measurement data) on the boundary, i.e. $\partial_\nu \textbf{u}|_{\partial\Omega\times(0,T)}$ (see for instance \cite{JinRundell2012IP_InverseTimeFrac1D}), our measurement map \eqref{FormalMapTime} is an average flux measurement on a portion of the boundary $\Gamma\subset\partial\Omega$, which is easier to measure in practical situations \cite{HydrologyPractical1,HydrologyPractical2}, yet the inverse problem is more severely ill-posed. This represents another distinctive aspect that sets our work apart from prior studies.

Moreover, similar to the space-nonlocal case, the highly nonlinear coupling in $\mathbf{F}$ significantly increases the complexity of the problem. Previous works on local parabolic systems use CGO solutions, as we have discussed above. But such CGO solutions typically involve an exponential-in-time term, which does not apply in the case of the fractional time derivative. Therefore, a new method is required, and we will provide the details in Section \ref{subsect:FTime}.

With regard to the recovery of the fractional orders $\bm{\beta}$, the only previous works addressing this issue are
\cite{RenHuangYamamoto2021JIIPTimeFracInversePbCoupled} and \cite{LiHuangLiu2024-TimeFracCoupledDiffusion-InversePbIdentifyExponent}. However, \cite{RenHuangYamamoto2021JIIPTimeFracInversePbCoupled} considered only a single time-fractional order, while \cite{LiHuangLiu2024-TimeFracCoupledDiffusion-InversePbIdentifyExponent} considered coupling in the fractional order. Instead, we consider  a system of equations coupled in the interaction functions $\mathbf{F}$ rather than the time-fractional term. Furthermore, our time-fractional term ${_{0}D}_t^{\bm{\beta}}$ consists of multiple time-fractional orders $\beta_{j,\cdot}$ with variable coefficients $b_{j,\cdot}(x)$. This time-fractional form has not been explored in existing literature. Additionally, we consider a different measurement map, given by the source-to-boundary measurement map $\Lambda^3$ in \eqref{FormalMapTime}. Hence, our method differs technically from previous works in this area.

The rest of the paper is organised as follows. In Section \ref{sect:prelims}, we present some preliminaries and statements. Section \ref{sect:linearzation} is devoted to the study of the method of high-order linearisation, which is necessary to treat the nonlinear coupling in the problem. The proofs of the unique determination for the coupled systems are provided in Sections \ref{sect:SpaceFracProof} and \ref{sect:TimeFracProof}. Finally, we give some applications of our results to various biological phenomena in Section \ref{sec:apply}.

\section{Preliminaries}\label{sect:prelims}

\subsection{Nonlocal Space Vector Calculus}

In this section, we briefly review the concepts of nonlocal calculus that are useful in what follows. We mainly focus on the nonlocal space operator $\mathcal{L}$, based on the references \cite{DuGunzburgerLehoucqZhouNonlocalDiffusion,DuGunzburgerLehoucqZhouNonlocalVectorCalculus}. For the time-fractional Caputo operator defined by \eqref{CaputoDer}, we refer readers to \cite{AtanganaBook,JinbookTimeFrac} for more details.

For vector functions  $\bm{\nu}(x,y),\ \bm{\alpha}(x,y):\mathbb{R}^d\times \mathbb{R}^d\rightarrow \mathbb{R}^k$ ($k\in\mathbb{N}$) such that $\bm{\alpha}$ is antisymmetric, i.e., $\bm{\alpha}(x,y)=-\bm{\alpha}(y,x)$, the nonlocal divergence operator $\mathcal{D}(\bm{\nu}):\mathbb{R}^d\rightarrow \mathbb{R}$ on $\bm{\nu}$ is defined as
\begin{equation}\label{DivDef}
\mathcal{D}(\bm{\nu})(x):=\int_{\mathbb{R}^d}(\bm{\nu}(x,y)+\bm{\nu}(y,x))\cdot \bm{\alpha}(x,y)\,dy,\quad \text{for }x\in \mathbb{R}^d.
\end{equation}
Its adjoint $\mathcal{D}^*:\mathbb{R}^d\times \mathbb{R}^d\rightarrow \mathbb{R}^k$ is given by
\begin{equation}\label{GradDef}
\mathcal{D}^*(u)(x,y):=-(u(y)-u(x)) \bm{\alpha}(x,y),\quad \text{for }x,y\in \mathbb{R}^d,
\end{equation}
for any scalar function $u(x):\mathbb{R}^d\rightarrow \mathbb{R}$. Therefore, the operator $-\mathcal{D}^*$ is commonly called the nonlocal gradient.

Using \eqref{DivDef} and \eqref{GradDef}, we have the following relation: For any second-order symmetric positive definite tensor $\bm{\Theta}(x,y)$, i.e. $\bm{\Theta}$ satisfies $\bm{\Theta}(x,y)=\bm{\Theta}(y,x)$  and $\bm{\Theta}=\bm{\Theta}^{T}$, 
\begin{equation}\label{OpDef}
\mathcal{D}(\bm{\Theta}\cdot\mathcal{D}^* u)(x):=-2\int_{\mathbb{R}^d}(u(y)-u(x))\bm{\alpha}(x,y)\cdot(\bm{\Theta}\cdot\bm{\alpha}(x,y))\,dy,\quad \text{for }x\in \mathbb{R}^d.
\end{equation}
Therefore, the operator $\mathcal{L}$ in \eqref{NonlocalLDef} can also be written as 
\begin{equation*}
-\mathcal{L}u=\mathcal{D}(\bm{\Theta}\cdot\mathcal{D}^* u):\mathbb{R}^d\rightarrow \mathbb{R},\quad \text{with } \gamma=\bm{\alpha}\cdot(\bm{\Theta}\cdot\bm{\alpha}).
\end{equation*}

For any open subset $\Omega\subset \mathbb{R}^d$, the corresponding interaction domain is defined by
\begin{equation}\label{InteractDomain}
\Omega_{\mathcal{I}}:=\{y\in \Omega^c:\bm{\alpha}(x,y)\neq 0\ \ \text{for}\ x \in\Omega\},
\end{equation}
so that $\Omega_{\mathcal{I}}$ consists of those points outside of $\Omega$ that interact with points in $\Omega$. Then, corresponding to the divergence operator $\mathcal{D}$ defined in \eqref{DivDef}, the nonlocal interaction operator $\mathcal{N}(\bm{\nu}):\mathbb{R}^d\rightarrow \mathbb{R}$ on $\bm{\nu}$ is defined by
\begin{equation}\label{NeumannDef}
\mathcal{N}(\bm{\nu})(x):=-\int_{\Omega\cup\Omega_I}(\bm{\nu}(x,y)+\bm{\nu}(y,x))\cdot\bm{\alpha}(x,y)\,dy,\quad\text{for }x\in\Omega_{\mathcal{I}}.
\end{equation}
Physically, the integral $\int_{\Omega_{\mathcal{I}}} \mathcal{N}(\bm{\nu})\,dx$ can be viewed as the nonlocal flux out of $\Omega$ into $\Omega_{\mathcal{I}}$ (refer to \cite{DuGunzburgerLehoucqZhouNonlocalVectorCalculus}).

With $\mathcal{D}$ and $\mathcal{N}$ defined in \eqref{DivDef} and \eqref{NeumannDef}, respectively, it is known that the nonlocal Gauss theorem \cite[Theorem 4.1]{DuGunzburgerLehoucqZhouNonlocalVectorCalculus} holds:
\begin{equation}
\int_{\Omega}\mathcal{D}(\bm{\nu})\,dx=\int_{\Omega_{\mathcal{I}}}\mathcal{N}(\bm{\nu})\,dx.
\end{equation}
Furthermore, the nonlocal Green's first identity \cite[Corollary 4.2]{DuGunzburgerLehoucqZhouNonlocalVectorCalculus} holds for scalar functions $u(x)$, $v(x)$:
\begin{equation}\label{Green}
\int_{\Omega}v\mathcal{D}(\bm{\Theta}\cdot\mathcal{D}^*u)\,dx-\int_{\Omega\cup\Omega_{\mathcal{I}}}\int_{\Omega\cup\Omega_{\mathcal{I}}}(\mathcal{D}^*v)\cdot(\bm{\Theta}\cdot\mathcal{D}^*u)\,dy\,dx
=\int_{\Omega_{\mathcal{I}}}v\mathcal{N}(\bm{\Theta}\cdot\mathcal{D}^*u)\,dx.
\end{equation}

\subsection{Properties of Nonlocal Operators}
Next, we show some properties of nonlocal operators that are essential for the proof. These are the maximum principles. We first begin with the maximum principle for \eqref{SpaceFracMainPb}.

\begin{lemma}[Weak Maximum Principle]\label{WMP} 
Assume $u\in C^{2,1}_{0}(\Omega\times(0,T])\cap C(\mathbb{R}^d\times[0,T])$. For $p\leq0$, if 
\[\partial_t u - \mathcal{L}u + \alpha\cdot\nabla u - p(x)u \geq 0\quad \text{ in }\Omega\times(0,T], \quad \quad u\geq 0 \quad \text{ in }\Omega^c\times(0,T],\] then 
\[u\geq 0\quad \text{ in }\Omega\times(0,T].\]
\end{lemma}

\begin{proof} Since $u$ is nonnegative outside $\Omega\times(0,T]$, assume by contradiction that $u(x,t)<0$ for some $(x,t)\in \Omega\times(0,T]$. Since $\Omega$ is a bounded open set, a minimal point $(x_0,t_0)$ is attained in $\Omega\times(0,T]$ and satisfies $u(x_0,t_0)<0$. Since $u\geq0$ outside $\Omega$, this means that $u(x_0,t_0)$ is a minimum in the whole space $\mathbb{R}^d\times(0,T]$, and we deduce that 
\[\left.\partial_t u\right|_{(x_0,t_0)}=\nabla u|_{(x_0,t_0)}=0.\] 

Set $r=dist(x_0,\partial\Omega)$ and let $B_r(x_0)$ denote the sphere with centre $x_0$ and radius $r$. Since $u(x_0,t_0)$ is a minimum in $\mathbb{R}^d\times(0,T]$, for any $y$ in $B_{2r}(x_0)\subset \mathbb{R}^d$, we have that $u(x_0,t_0)-u(y,t_0)\leq 0$. On the other hand, if $y\in \mathbb{R}^d\setminus B_{2r}(x_0)\subset\Omega^c$, then $|x-y|\geq|y-x_0|-|x-x_0|\geq r$. Furthermore, by the assumptions of the lemma, $u(y,t_0)\geq 0$. 

Combining these two estimates, we have
\begin{align*}
0&\leq\left.\left(\partial_t u-\mathcal{L}u+ \alpha\cdot\nabla u-p(x)u\right)\right|_{(x_0,t_0)}\\
&=\left.\partial_t u \right|_{(x_0,t_0)}+2\int_{\mathbb{R}^d}(u(x_0,t_0)-u(y,t_0)) \gamma(x_0,y)\,dy+\alpha\cdot\nabla u|_{(x_0,t_0)}-p(x_0)u(x_0,t_0)\\
&\leq 2\sum_{i=1}^N\int_{\mathbb{R}^d}\frac{\gamma_*(u(x_0,t_0)-u(y,t_0))}{|x_0-y|^{n+2s_i}}\,dy\\
&=2\gamma_*\sum_{i=1}^N\int_{B_{2r}(x_0)}\frac{u(x_0,t_0)-u(y,t_0)}{|x_0-y|^{n+2s_i}}\,dy+2\gamma_*\sum_{i=1}^N\int_{\mathbb{R}^d\setminus B_{2r}(x_0)}\frac{u(x_0,t_0)-u(y,t_0)}{|x_0-y|^{n+2s_i}}\,dy\\
&\leq 2\gamma_*\sum_{i=1}^N\int_{\mathbb{R}^d \setminus B_{2r}(x_0)}\frac{u(x_0,t_0)}{|x_0-y|^{n+2s_i}}\,dy<0
\end{align*}
by \eqref{KernelAssump1}. 
This is contradictory, so it must be that $u\geq 0$ in $\Omega\times(0,T)$.
\end{proof}

\begin{lemma}[Strong Maximum Principle]\label{StrongMaxPrinciple}
    Assume $u\in C^{2,1}_{0}(\Omega\times(0,T])\cap C(\mathbb{R}^d\times[0,T])$. Suppose $p\leq0$ and 
    \[\partial_t u - \mathcal{L}u + \alpha\cdot\nabla u - p(x)u \geq 0\quad \text{ in }\Omega\times(0,T], \quad \quad u\geq 0 \quad \text{ in }\Omega^c\times(0,T].\] 
    Then there exists $t_0\in(0,T]$ such that $u(x,t_0)>0$ in $\Omega\times\{t_0\}$, unless $u\equiv0$ in $\mathbb{R}^d\times\{t_0\}$.
\end{lemma}

\begin{proof}
    We already know that $u\geq0$ in $\mathbb{R}^d\times(0,T]$, since the assumption states that $u\geq0$ outside $\Omega$, while Lemma \ref{WMP} gives the result in $\Omega$. Hence, if $u\equiv0$ at some time $t_0\in(0,T]$, then we are done. Otherwise, suppose on the contrary that for all $t\in(0,T]$, there exists $x_0\in\Omega$ such that $u(x_0,t)=0$. Then $\left.\partial_t u \right|_{(x_0,t)}=0$ for all $t\in(0,T)$. Furthermore, since $u(x,t)\geq0$ in $\Omega\times(0,T]$ and $u(x_0,t)=0$, $u$ is either constant near $x_0$ or is a minimal point, so $\nabla u|_{(x_0,t_0)}=0$ for some $t_0\in(0,T]$. This gives that
    \begin{align*}
        0&\leq\left.\left(\partial_t u -\mathcal{L}u+\alpha\cdot\nabla u-p(x)u\right)\right|_{(x_0,t_0)}\\
        &=\left.\partial_t u \right|_{(x_0,t_0)}+2\int_{\mathbb{R}^d}(u(x_0,t_0)-u(y,t_0)) \gamma(x_0,y)dy+\alpha\cdot\nabla u|_{(x_0,t_0)}-p(x_0)u(x_0,t_0)\\
        &=-2\int_{\mathbb{R}^d}u(y,t_0)\gamma(x_0,y)dy\leq 0
    \end{align*}
    by \eqref{KernelAssump1}. Therefore, it must be that $u\equiv 0$ in $\Omega\times\{t_0\}$, and the conclusion is established.

\end{proof}

Next, similar to the nonlocal space-fractional case, the corresponding maximum principle for the multi-term time-fractional case is also known, by combining Theorems 2 and 3 of \cite{Luchko2011MultiTimeFracMaxPrinciple}:

\begin{lemma}[Strong Maximum Principle]\label{StrongMaxPrincipleTime}
    Assume $u\in C^{2,1}_{0}(\Omega\times(0,T])\cap C(\bar{\Omega}\times[0,T])$ and $p\leq0$, such that 
    \begin{equation}
\sum_{j=1}^{B}b_j({_{0}D}_t^{\beta_j}u)-d(x)\Delta u + \alpha(x)\cdot\nabla u -p(x,t)u \geq 0, \quad \text{in } \Omega\times(0,T],
\end{equation}
and $u\geq0$ in $\partial\Omega\times(0,T]$. Then $u\geq0$ in $\Omega\times(0,T)$, unless $u$ vanishes identically.
\end{lemma}

\subsection{Main Results}\label{subsect:MainResDef}
With these properties, we can now state our main results. We first define our measurement maps and admissible sets.

Consider the initial-exterior value problem
\begin{equation}\label{DNmapProb}
	\begin{cases}
		\partial_t u -\mathcal{L}u+\alpha\cdot\nabla u=pu+q&  \text{ in }\Omega\times(0,T],\\
		u=0 &\text{ in }\Omega^c\times(0,T],\\
		u(x,0)=0 & \text{ in }\Omega.
	\end{cases}
\end{equation}

We introduce the source-to-boundary map $\Lambda^1_{p}$ related to \eqref{DNmapProb} given by
	\begin{equation}\label{DNMapSpace}
        \Lambda^1_{p}(q)=\left.(-\mathcal{L}+\alpha\cdot\nabla)u\right|_{\Omega_a\times(0,T)}\in L^2(0,T;H^{-s}(\Omega^c)),
	\end{equation}
where $u:\mathbb{R}^d\times (0,T)\to\mathbb{R}$ is the solution to \eqref{DNmapProb}, 
defined in the distributional sense, i.e. $\langle\Lambda^1_{p}(q),h\rangle$ for given Dirichlet data $h$ defined in the admissible region $\Omega_a\subset\Omega^c$.
Therefore, it can be seen that $h$ characterises the basic properties of measure instrument.

The coupling in the systems \eqref{SpaceFracMainPb} and \eqref{TimeMainPb0} are given in the nonlinear interaction functions $\mathbf{F}(x,t,\mathbf{u})$. This coupling is described by the following admissible class:
\begin{definition}[Admissible class $\mathcal{A}$]\label{adm1}
For $\mathbf{F}(x,t,\mathbf{u})=(F_1(x,t,\mathbf{u}),F_2(x,t,\mathbf{u}),\dots,F_M(x,t,\mathbf{u}))$, we define $F_i(x,t,\mathbf{u})$ to be an element of the admissible set $\mathcal{A}$ if it can be inferred that $F_i(x,t,\mathbf{u})$ can be expressed as a power series expansion in the following form:
\[
F_i(x,t,\mathbf{u})=\sum_{\substack{k_1+\cdots+k_M\geq2\\k_i\geq1}}^{\infty}F^{(k_1k_2\cdots k_M)}_iu_1^{k_1}u_2^{k_2}\cdots u_M^{k_M},
\]
where $F^{(k_1k_2\cdots k_M)}_i$ is a constant.
\end{definition}

Then, our main result for the coupled nonlinear nonlocal parabolic system is as follows:

\begin{theorem}\label{SpaceFracThm}
Let $\mathbf{F}\in\mathcal{A}$ and $\mathbf{q}^n=(q_1^n,\dots,q_M^n)$ be such that $\{q_i^n\}_{n=1}^{\infty}$ is a complete set in $L^2(\Omega)$ for each $i=1,\dots,M$. Take $\mathbf{h}\in [C_c^{2,1}(\Omega_a\times(0,T))]^M$ to be given nonzero nonnegative functions. Assume that $p_i(x)\in C(\overline\Omega)$, $p_i\leq 0$ on $\Omega$ for every $i$. 
Let $\mathbf{u}^{n}(x,t;\mathbf{p}^1,\mathbf{F}^1)$, $\mathbf{u}^{n}(x,t;\mathbf{p}^2,\mathbf{F}^2)$  be the bounded classical solutions of problem \eqref{SpaceFracMainPb} corresponding to the potentials $\mathbf{p}^k$ and interaction functions $\mathbf{F}^k$ ($k=1,2$) respectively. If 
\begin{equation}\label{condtionSpace}
    \langle\Lambda^1_{\mathbf{p}^1,\mathbf{F}^1},\mathbf{h}\rangle=\langle\Lambda^1_{\mathbf{p}^2,\mathbf{F}^2},\mathbf{h}\rangle
\end{equation}
then 
\[\mathbf{p}^1=\mathbf{p}^2\quad\text{ in }\Omega\quad \text{ and }\quad \mathbf{F}^1=\mathbf{F}^2.\]
\end{theorem}

Here, we use the notation $\mathbf{u}^{n}(x,t;\mathbf{p},\mathbf{F})$ for the solution to emphasise its dependence on the functions $\mathbf{p}$ and $\mathbf{F}$ and the inject sources $\mathbf{q}$.

\begin{remark}
    This result is weaker than that obtained when the operator $\mathcal{L}$ is replaced with the classical Laplacian (c.f. Theorem 4.1 of \cite{ding2023determining}). This is because the strong maximum principle (Lemma \ref{StrongMaxPrinciple}) only holds on $t_0\in(0,T]$ for the nonlocal case, so $p_i$ cannot depend on $t$. This is an effect of nonlocality.
\end{remark}

In general, no good Neumann boundary condition definition is known for mixed local-nonlocal operators involving a fractional or nonlocal second-order derivative and a first order local gradient operator. The only known Neumann boundary conditions for mixed operators involve two second-order operators, as in \cite{DipierroLippiValdinoci2022-NeumannConditionMultipleFracLap,DipierroLippiSportelliValdinoci2024-NeumannConditionMultipleFracLap,DipierroLippiSportelliValdinoci2024-NeumannConditionMultipleFracLap2}. Yet, in the case where the operator is completely nonlocal, i.e. when $\bm{\alpha}\equiv\mathbf{0}$, it is known that 
\[\langle\Lambda^1_{p_i}(q),h\rangle|_{\Omega_a}=\int_{\Omega_a\times(0,T)}\mathcal{N}(\bm{\Theta}\cdot \mathcal{D}^*(u^{m,n}_i(x,t;p_i)))h(x,t)\,dxdt.\] In this situation, we have the corresponding result:

\begin{corollary} 
Let $\mathbf{F}\in\mathcal{A}$ and $\mathbf{q}^n=(q_1^n,\dots,q_M^n)$ be such that $\{q_i^n\}_{n=1}^{\infty}$ is a complete set in $L^2(\Omega)$ for each $i=1,\dots,M$. Take $\mathbf{h}\in [C_c^{2,1}(\Omega_a\times(0,T))]^M$ to be given nonzero nonnegative functions. Assume that $p_i(x)\in C(\overline\Omega)$, $p_i\leq 0$ on $\Omega$ for every $i$. 
Let $\mathbf{u}^{n}(x,t;\mathbf{p}^1,\mathbf{F}^1)$, $\mathbf{u}^{n}(x,t;\mathbf{p}^2,\mathbf{F}^2)$  be the bounded classical solutions of problem \eqref{SpaceFracMainPb} corresponding to the potentials $\mathbf{p}^k$ and interaction functions $\mathbf{F}^k$ ($k=1,2$) respectively. If 
\begin{equation}\label{condtionSpaceNonmixed}
    \int_{\Omega_a\times(0,T)}\mathcal{N}(\bm{\Theta}\cdot \mathcal{D}^*(u^{n}_i(x,t;p_i^1)))h(x,t)\,dxdt=\int_{\Omega_a\times(0,T)}\mathcal{N}(\bm{\Theta}\cdot \mathcal{D}^*(u^{n}_i(x,t;p_i^2)))h(x,t)\,dxdt
\end{equation}
then 
\[\mathbf{p}^1=\mathbf{p}^2\quad\text{ in }\Omega \quad \text{ and }\quad \mathbf{F}^1=\mathbf{F}^2.\]
\end{corollary}

In the case where the operator $\mathcal{L}$ is local and we consider nonlocality in the time derivative, we have a similar result for the time-fractional diffusion system.

We first introduce the source-to-boundary map $\Lambda^2_{p}$ related to \begin{equation}\label{DNmapProbTime}
	\begin{cases}
		{_{0}D}_t^{\bm{\beta}}\mathbf{u} -\mathbf{d}\Delta \mathbf{u}+\bm{\alpha}\cdot\nabla \mathbf{u}=\mathbf{p}\mathbf{u}+\mathbf{F}(x,t,\mathbf{u})+\mathbf{q}&  \text{ in }\Omega\times(0,T],\\
		\mathbf{u}=\mathbf{0} &\text{ in }\partial\Omega\times(0,T],\\
		\mathbf{u}(0)=\mathbf{0} & \text{ in }\Omega.
	\end{cases}
\end{equation}
given by
	\begin{equation}\label{DNMapTime}
        \Lambda^2_{\mathbf{p},\mathbf{F}}(\mathbf{q})=\left.(-\mathbf{d}\Delta+\bm{\alpha}\cdot\nabla)\mathbf{u}\right|_{\Gamma\times(0,T)},
	\end{equation}
where $u_i:\mathbb{R}^d\times (0,T)\to\mathbb{R}$ is the solution to \eqref{TimeMainPb0}, 
defined in the distributional sense, i.e. $\langle\Lambda^2_{\mathbf{p},\mathbf{F}}(\mathbf{q}),\mathbf{h}\rangle$ for given Dirichlet data $\mathbf{h}=(h_1,\dots,h_M)$ defined on $\Gamma\subset\partial\Omega$.

We also introduce the source-to-interior map $\Lambda_{\bm{\beta}}^3(\mathbf{q})$ associated to \eqref{DNmapProbTime}, given by the observation of $\mathbf{u}$ at $\{x_0\}\times (0,T)$, i.e. 
\begin{equation}
    \Lambda^3_{\bm{\beta}}(\mathbf{q})=\mathbf{u}(x_0,t)|_{t\in(0,T)}.
\end{equation}

Then, our main result is as follows:
\begin{theorem}\label{TimeFracThm}
Let $\mathbf{F}\in\mathcal{A}$ and $\mathbf{q}^n=(q_1^n,\dots,q_M^n)$ be such that $\{q_i^n\}_{n=1}^{\infty}$ is a complete set in $L^2(\Omega\times(0,T))$ for each $i=1,\dots,M$. For $\Gamma\subseteq\partial\Omega$, take $\mathbf{h}\in [C_c(\Gamma\times(0,T))]^M$ to be given nonzero nonnegative functions. Assume that $p_i(x,t)\in C(\overline\Omega\times[0,T])$, $p_i\leq 0$ on $\Omega\times(0,T]$ for every $i$. 
Let $\mathbf{u}^{n}(x,t;\mathbf{p}^1,\mathbf{F}^1,\bm{\beta}^1)$, $\mathbf{u}^{n}(x,t;\mathbf{p}^2,\mathbf{F}^2,\bm{\beta}^2)$  be the bounded classical solutions of problem \eqref{TimeMainPb0} corresponding to the potentials $\mathbf{p}^k$ and interaction functions $\mathbf{F}^k$ ($k=1,2$) respectively. 
\begin{enumerate}[label=(\arabic*)]
    \item If 
\begin{equation}\label{condtionTime}
    \Lambda^2_{\mathbf{p}^1,\mathbf{F}^1}=\Lambda^2_{\mathbf{p}^2,\mathbf{F}^2}
\end{equation}
then 
\begin{equation}\mathbf{p}^1=\mathbf{p}^2\quad\text{ in }\Omega\times(0,T) \quad\text{ and }\quad\mathbf{F}^1=\mathbf{F}^2.\end{equation}
\item If 
\begin{equation}\label{condtionTimeOrder}
    \Lambda^3_{\bm{\beta}^1}=\Lambda^3_{\bm{\beta}^2}
\end{equation}
then 
\begin{equation}\bm{\beta}^1=\bm{\beta}^2.\end{equation}
\item If 
\begin{equation}
    \Lambda^2_{\mathbf{p}^1,\mathbf{F}^1}=\Lambda^2_{\mathbf{p}^2,\mathbf{F}^2}\quad \text{ and }\quad \Lambda^3_{\bm{\beta}^1}=\Lambda^3_{\bm{\beta}^2}
\end{equation}
then 
\begin{equation}\mathbf{p}^1=\mathbf{p}^2\quad\text{ in }\Omega\times(0,T) \quad\text{ and }\quad\mathbf{F}^1=\mathbf{F}^2\quad \text{ and }\quad \bm{\beta}^1=\bm{\beta}^2.\end{equation}

Here, we are able to recover $\mathbf{p}$, $\mathbf{F}$ and $\bm{\beta}$ simultaneously. 
\end{enumerate}
\end{theorem}

\section{High Order Linearisation}\label{sect:linearzation}

We will prove the results by considering the different orders of expansion of $F_i\in\mathcal{A}$, similar to the method of high-order linearisation. This is the focus of this section. We will show the process for the coupled nonlocal nonlinear parabolic system, and the time-fractional coupled nonlinear diffusion system follows similarly.

Assume that the solutions of \eqref{SpaceFracMainPb} and \eqref{TimeMainPb0} are bounded classical solutions. 
Consider source inputs of the form 
\begin{equation}\label{m0}
\textbf{q}(x,t;\epsilon)=\sum_{l=1}^K\epsilon_l\textbf{g}_l(x,t),\quad (x,t)\in \Omega\times(0,T],
\end{equation}
where $\epsilon=(\epsilon_1,\epsilon_2,\dots,\epsilon_K)\in \mathbb{R}^K_+$ with $|\epsilon|=\sum_{l=1}^K |\epsilon_l|$ sufficiently small. 

Given \eqref{m0}, we can observe that $\bm{u}=\bm{0}$ is a solution to both \eqref{SpaceFracMainPb} and \eqref{TimeMainPb0} when $\epsilon=0$. Therefore, we define $\bm{u}(x,t;0)=\bm{0}$ to be the solution of the systems when $\epsilon=0$.

Let $S:\textbf{q}\mapsto\textbf{u}$ be the solution operator of \eqref{SpaceFracMainPb} or \eqref{TimeMainPb0}. Then there exists a bounded linear operator $\mathcal{B}$ from $\mathcal{H}:=[B_{\delta}( C^{2,1}_{0}(\Omega\times(0,T])\cap C(\bar{\Omega})]^M$ to $[C^{2,1}_{0}(\Omega\times(0,T])\cap C(\bar{\Omega}]^M$ such that
\begin{equation}
	\lim_{\norm{\textbf{q}}_{\mathcal{H}}\to0}\frac{\|S(\textbf{q})-S(\bm{0})- \mathcal{B}(\textbf{q})\|_{[C^{2,1}_{0}(\Omega\times(0,T])\cap C(\bar{\Omega}]^M}}{\norm{\textbf{q}}_{\mathcal{H}}}=0.
\end{equation} 

Now we consider $\varepsilon_l=0$ for $l=2,\dots,K$ and fix $g_1$. 
Notice that if $F_i\in\mathcal{A}$, then each $F_i$ depends on each $u_j$ locally, for $j=1,\dots,M$. 
Then it is easy to check that $\mathcal{B}(\textbf{q})\Big|_{\varepsilon_1=0}$ is the solution map of the following system which is called the first-order linearisation system:

\begin{equation}\label{linear1main}
\begin{cases}
\partial_t \textbf{u}^{(1)}-\mathcal{L} \textbf{u}^{(1)}+\bm{\alpha}\cdot\nabla \textbf{u}^{(1)}=\textbf{p}(x,t)\textbf{u}^{(1)}+\textbf{g}_1 &\text{ in }\  \Omega\times(0,T],\\
\textbf{u}^{(1)}=0&\text{ on }\  \Omega^c\times(0, T],\\
\textbf{u}^{(1)}(x,0)=0,&\text{ in }\ \mathbb{R}^d,\\
\textbf{u}^{(1)} \geq 0 &\text{ in }\ \bar{\Omega}\times[0,T],
\end{cases}
\end{equation}

 In the following, we define
\begin{equation}\label{eq:ld1}
 \textbf{u}^{(1)}:=\mathcal{B}(\textbf{q})\Big|_{\varepsilon_1=0}. 
 \end{equation}
For notational convenience, we write
\begin{equation}\label{eq:ld2}
 \textbf{u}=\partial_{\varepsilon_1} \textbf{u}(x,t;\varepsilon)|_{\varepsilon=0}.
\end{equation}
We shall utilise such notations in our subsequent discussion in order to simplify the exposition and their meaning should be clear from the context. In a similar manner, we can define $\textbf{u}^{(l)} := \left. \partial_{\varepsilon_l} \textbf{u} \right\rvert_{\varepsilon = 0}$ for $l=2,\dots,K$, and obtain a sequence of similar systems.

For the higher orders, we consider
\[\textbf{u}^{(1,2)}:= \left. \partial_{\varepsilon_1} \partial_{\varepsilon_2} \textbf{u} \right\rvert_{\varepsilon = 0}.\]
Similarly, $\textbf{u}^{(1,2)}$ can be viewed as the output of the second-order Fr\'echet derivative of $S$ at a specific point. By following similar calculations in deriving \eqref{linear1main}, one can show that the second-order linearisation is given, for $u_i$ ($i=1,\dots,M$), as follows:
\begin{equation}\label{linear2main}
\begin{cases}
\partial_t u^{(1,2)}_i-\mathcal{L} u_i^{(1,2)}+\alpha\cdot\nabla u_i^{(1,2)}=p_i(x,t)u_i^{(1,2)}+2F^{(0\cdots2\cdots 0)}_iu_i^{(1)}u_i^{(2)}\\\qquad\qquad\qquad\qquad\qquad\qquad\qquad\qquad\qquad+\sum\limits_{j\neq i}^KF^{(0\cdots1\cdots1\cdots0)}_i(u_i^{(1)}u_j^{(2)}+u_i^{(2)}u_j^{(1)}) &\text{ in }\  \Omega\times(0,T],\\
\textbf{u}^{(1,2)}=0&\text{ on }\  \Omega^c\times(0, T],\\
\textbf{u}^{(1,2)}(x,0)=0,&\text{ in }\ \mathbb{R}^d,\\
\textbf{u}^{(1,2)} \geq 0 &\text{ in }\ \bar{\Omega}\times[0,T],
\end{cases}
\end{equation}

Inductively, for $l\in \mathbb{N}$, we consider
\[\textbf{u}^{(1,2,\dots,l)}:= \left. \partial_{\varepsilon_1} \partial_{\varepsilon_2} \cdots \partial_{\varepsilon_l} \textbf{u} \right\rvert_{\varepsilon = 0},\] and obtain a sequence of parabolic systems.

In the time-fractional case, we have the following first-order linearisation system: 
\begin{equation}\label{TimeLinear1}
\begin{cases}
{_{0}D}_t^{\bm{\beta}}\mathbf{u}^{(l)}-\mathbf{d}\Delta \mathbf{u}^{(l)}+\bm{\alpha}\cdot\nabla \mathbf{u}^{(l)}=\mathbf{p}\mathbf{u}^{(l)}+\mathbf{g}_1&\text{ in }\  \Omega\times(0,T],\\
 \mathbf{u}^{(l)}=0&\text{ on }\  \partial\Omega\times(0, T),\\
\mathbf{u}^{(l)}(x,0)=0&\text{ in }\  \Omega,\\
\mathbf{u}^{(l)}\geq 0&\text{ in }\  {\bar{\Omega}\times[0,T]},
\end{cases}
\end{equation}
and the other linearised systems similarly.

\section{Proof of Theorem \ref{SpaceFracThm}}\label{sect:SpaceFracProof}
In the following proof, we will always show the case for $u_1$, and the results for $u_i$ ($i=2,\dots,M$) hold similarly. Therefore, we will drop all subscript 1's, when it is clear from the context.

\subsection{Determining the Potentials $\mathbf{p}$}
For the nonlocal parabolic systems, we choose the input function $q$ such that 
\[q(x,t;\epsilon)=\sum_{l=1}^K\epsilon_lg_l(x,t),\quad (x,t)\in \Omega\times(0,T],\]
where \[g_1(x,t)=\varphi^n(x)V(t)>0, \quad \varphi(x) \text{ is a complete set in }L^2(\Omega),\]
and $g_l>0$ for $l=2,\dots,K$.

Then, the first order linearised system is given by 
\begin{equation}\label{linear1space}
\begin{cases}
\partial_t u^{(1)}-\mathcal{L} u^{(1)}+\alpha\cdot\nabla u^{(1)}=pu^{(1)}+g_1 &\text{ in }\  \Omega\times(0,T],\\
u^{(1)}=0&\text{ on }\  \Omega^c\times(0, T],\\
u^{(1)}(x,0)=0,&\text{ in }\ \mathbb{R}^d,\\
u^{(1)} \geq 0 &\text{ in }\ \bar{\Omega}\times[0,T],
\end{cases}
\end{equation}

Set $u_k$ to be the solution of \eqref{DNmapProb} corresponding to the input functions $\varphi^n$ and $p^k$, $k=1,2$.
For a given nonzero nonnegative function $h\in C_c^{2,1}(\Omega_a\times(0,T])$, we set $h_0=h$ on $\Omega_a\times(0,T]$ and $h_0=0$ on $(\mathbb{R}^d\setminus(\Omega\cup\Omega_a))\times(0,T]$, and introduce the function $ w(x,t;p)$ as the solution of the following adjoint problem
\begin{equation}\label{AdjointPb}
\begin{cases}

-\partial_t w -\mathcal{L}w+\alpha\cdot\nabla w = pw,\quad&\text{ in }\Omega\times(0,T],\\
w=h_0(x,t),\quad&\text{ on }\Omega^c\times(0,T],\\
w=0,\quad&\text{ in }\Omega\times\{T\}.\\

\end{cases}
\end{equation}
Then it is known that $w$ admits H\"older regularity \cite{CaffarelliVasseur2010AnnMathFracSQG,SilvestreFracDriftHolder}, and therefore uniformly bounded for a bounded domain $\Omega\times(0,T]$ and $\Omega_a\times(0,T]$.

Using the equation for $u_1$ in \eqref{linear1space}, we compute
\begin{align*}
&\int_{0}^T \int_{\Omega}\varphi^n(x)V(t)w(x,t;p^1)\,dxdt\\
=&\int_{0}^T\int_{\Omega}\left(\partial_t u_1- \mathcal{L}u_1+\alpha\cdot\nabla u_1-p^1u_1\right)w(x,t;p^1)\,dxdt\\
=&\int_{0}^T\int_{\Omega}\left(-\partial_t w -\mathcal{L}w+\alpha\cdot\nabla w-p^1w\right)u_1\,dxdt +\langle\Lambda^1_{p^1}(h),u_1\rangle|_{\Omega_a}\\
=&\langle\Lambda^1_{p^1}(h),u_1\rangle|_{\Omega_a},
\end{align*}
by the construction of $w$ in \eqref{AdjointPb}. 
Similarly, for the same equation for $u_2$ with potential $p^2$, we have 
\begin{equation*}
\int_{0}^T \int_{\Omega}\varphi^n(x)V(t)w(x,t;p^2)dxdt=\langle\Lambda^1_{p^2}(h),u_2\rangle|_{\Omega_a},
\end{equation*}
By \eqref{condtionSpace}, we obtain the equality 
\[\int_{0}^T \int_{\Omega}\varphi^n(x)V(t)w(x,t;p^1)\,dxdt=\int_{0}^T \int_{\Omega}\varphi^n(x)V(t)w(x,t;p^2)\,dxdt.\]
The completeness of $\{\varphi^n\}_n$ then implies 
\begin{equation}\label{Veq0}\int_0^TV(t)w(x,t;p^1)\,dt=\int_0^TV(t)w(x,t;p^2)\,dt\quad \text{ in }\Omega.\end{equation}

Next, given any test function $v$, we consider $V=v$ and $V=\partial_t v$, so that we have 
\begin{equation}\label{Veq1}
    \int_0^Tv(t)w(x,t;p^1)\,dt=\int_0^Tv(t)w(x,t;p^2)\,dt\quad \text{ in }\Omega,
\end{equation}
and 
\begin{equation}\label{Veq2}
    \int_0^T\partial_t v w(x,t;p^1)\,dt=\int_0^T\partial_t v w(x,t;p^2)\,dt\quad \text{ in }\Omega.
\end{equation}
In particular, we can apply the space operator $-\mathcal{L}+\alpha\cdot\nabla$ to the first case \eqref{Veq1}, to obtain
\begin{equation}\label{Veq3}
    \int_0^Tv(t)(-\mathcal{L}+\alpha\cdot\nabla)w(x,t;p^1)\,dt=\int_0^Tv(t)(-\mathcal{L}+\alpha\cdot\nabla)w(x,t;p^2)\,dt\quad \text{ in }\Omega.
\end{equation}

Consequently, multiplying \eqref{AdjointPb} by $v$ and integrating by parts over $(0,T)$, 
we find
\begin{equation}
\begin{split}
-\int_{0}^T\partial_t w(x,t;p^1) v\,dt+\int_0^T(-\mathcal{L}+\alpha\cdot\nabla)w(x,t;p^1)v\,dt=\int_0^T p^1(x)w(x,t;p^1)v\,dt;\\
-\int_{0}^T\partial_t w(x,t;p^2) v\,dt+\int_0^T(-\mathcal{L}+\alpha\cdot\nabla)w(x,t;p^2)v\,dt=\int_0^T p^2(x)w(x,t;p^2)v\,dt.
\end{split}
\end{equation}
Taking the difference of these two equations, and making use of the equalities \eqref{Veq1}-- \eqref{Veq3}, we have 
\[(p^1(x)-p^2(x))\int_0^T w(x,t;p^2)v\,dt=0,\]
so for convenience, we can drop the dependence on $p^k$, $k=1,2$.

Since $w$ satisfies \eqref{AdjointPb} with nonzero $h$, by the strong maximum principle, there exists $t_0\in(0,T]$ and $\epsilon_0>0$ such that 
\[\int_0^T w(x,t;p^2)v\,dt\geq\int_{t_0-\epsilon_0}^{t_0} w(x,t;p^2)v\,dt>0.\] 
Thus, $p^1(x)=p^2(x)$ in $\Omega$, and the proof is complete for the potentials.

\subsection{Determining the Interaction Functions $\mathbf{F}$}\label{sect:ProofSourceSpace}
We first note that having determined $\mathbf{p}$, the first order solutions $\mathbf{u}^1$ can be uniquely determined from \eqref{linear1space}. Therefore, $\mathbf{u}^1_1=\mathbf{u}^1_2=:\mathbf{u}^1$

Next, we consider the second-order linearisation. Once again, we consider the equation for $u_{1,k}$ (and drop all subscripts `1' as before), $k=1,2$, given by 
\begin{equation}\label{linear2space}
\begin{cases}
\partial_t u^{(1,2)}_k-\mathcal{L} u^{(1,2)}_k+\alpha\cdot\nabla u^{(1,2)}_k=p(x,t)u^{(1,2)}_k+2F^{(20\cdots 0)}_ku^{(1)}u^{(2)}\\
\qquad\qquad\qquad\qquad\qquad\qquad\qquad\qquad\qquad+\sum\limits_{j=2}^KF^{(10\cdots1\cdots0)}_k(u^{(1)}u_j^{(2)}+u^{(2)}u_j^{(1)}) &\text{ in }\  \Omega\times(0,T],\\
u^{(1,2)}_k=0&\text{ on }\  \Omega^c\times(0, T],\\
u^{(1,2)}_k(x,0)=0,&\text{ in }\ \mathbb{R}^d.
\end{cases}
\end{equation}
Taking the difference of the two equations for $u_k$ corresponding to $F_k$, we have
\begin{equation}
\begin{cases}
\partial_t \tilde{u}-\mathcal{L} \tilde{u}+\alpha\cdot\nabla \tilde{u}=p(x,t)\tilde{u}+2(F^{(20\cdots 0)}_1-F^{(20\cdots 0)}_2)u^{(1)}u^{(2)}\\
\qquad\qquad\qquad\qquad\qquad+\sum\limits_{j=2}^K(F^{(10\cdots1\cdots0)}_1-F^{(10\cdots1\cdots0)}_2)(u^{(1)}u_j^{(2)}+u^{(2)}u_j^{(1)}) &\text{ in }\  \Omega\times(0,T],\\
\tilde{u}=0&\text{ on }\  \Omega^c\times(0, T],\\
\tilde{u}(x,0)=0,&\text{ in }\ \mathbb{R}^d,
\end{cases}
\end{equation}
where $\tilde{u}=u_1^{(1,2)}-u_2^{(1,2)}$.

Multiplying both sides of the equation by $w$ in \eqref{AdjointPb} and integrating by parts, we have that 
\begin{multline}\label{SpaceEq}2(F^{(20\cdots 0)}_1-F^{(20\cdots 0)}_2)\int_0^T\int_\Omega u^{(1)}u^{(2)}w\,dxdt\\+\sum\limits_{j=2}^K(F^{(10\cdots1\cdots0)}_1-F^{(10\cdots1\cdots0)}_2)\int_0^T\int_\Omega (u^{(1)}u_j^{(2)}+u^{(2)}u_j^{(1)})w\,dxdt=0.\end{multline} 
Applying the maximum principle on $u^{(l)},u_j^{(l)}$ and $w$ for $h>0$, we have that the integrals 
\begin{equation}\label{MaxPrincipleSpaceIntegrals}\int_0^T\int_\Omega u^{(1)}u^{(2)}w\,dxdt,\quad\int_0^T\int_\Omega (u^{(1)}u_j^{(2)}+u^{(2)}u_j^{(1)})w\,dxdt>0.\end{equation}

Suppose all except one of the $F^{(\cdot)}$'s are fixed and known, say $\hat{F}=F^{(20\cdots0)}$ or $F^{(10\cdots1\cdots0)}$. Then, from \eqref{SpaceEq}, by \eqref{MaxPrincipleSpaceIntegrals}, we have that $\hat{F}_1=\hat{F}_2$. 

Similarly, this holds for $u_i$, $i=2,\dots,M$, and we have the unique identifiability for $\mathbf{F}^{(1)}$.

Observe that, when we consider the $l$-th order of linearisation, the terms of $\mathbf{u}$ attached to $\mathbf{F}$ on the right-hand-side (as in \eqref{linear2space} for the 2nd order of linearisation) depend only on the lower order linearised solutions of $u$. Therefore, by mathematical induction, we derive the unique identifiability result for the $l$-th order Taylor coefficient of $F$, and the result is proved.

\begin{remark}\label{rmk:SpaceExtNonConstant}
    It may be possible to extend this result to obtain unique identifiability for general $F_i$ with Taylor coefficients $F_i^{(\cdot)}(x,t)$ depending on both the space $x$ and time $t$ variables, by making use of unique continuation and Runge approximation properties. However, such properties are not yet known for our measurement map $\Lambda^1$. Previous results have only been obtained for the Dirichlet-to-Neumann measurement map (see, for instance \cite{RulandSalo2020FracHeatUCP}).
    % In the case where $\bm{\alpha}=\mathbf{p}=\bm{0}$ and the nonlocal operator $\mathcal{L}$ has a Caffarelli-Silvestre extension, it is possible to improve the result on $\mathbf{F}$ to obtain unique identifiability for general $F_i$ with Taylor coefficients $F_i^{(\cdot)}(x,t)$ depending on both the space $x$ and time $t$ variables. This result follows the approach in many works on nonlocal inverse problems, by making use of the %Runge density result on $w$ in \eqref{AdjointPb}
    % unique continuation property of 
    % %https://arxiv.org/pdf/2312.15607
    % %https://arxiv.org/pdf/1708.06300
\end{remark}

\section{Proof of Theorem \ref{TimeFracThm}}\label{sect:TimeFracProof}

Next, we give the proof for the time-fractional case with the operator $\mathcal{L}$ being local. 
Once again, we will only show the case for $u_1$, and the results for $u_i$ ($i=2,\dots,M$) hold similarly. Therefore, we will drop all subscript 1's, when it is clear from the context.

We first begin with a simple observation:

\begin{lemma}
    Deriving the uniqueness of $\tilde{\textbf{p}}$ in 
    \begin{equation}\label{TimeMainPb2}
\begin{cases}
{_{0}D}_t^{\bm{\beta}}\mathbf{u}-\mathbf{d}\Delta_y \mathbf{u}+\bm{\alpha}\cdot\nabla_y \mathbf{u}=\tilde{\mathbf{p}}(y,t)\mathbf{u}+\tilde{\mathbf{F}}(y,t,\mathbf{u})+\mathbf{q}e^{-\frac{\bm{\alpha}x}{2\sqrt{\mathbf{d}}}} &\text{ in }\  \Omega\times(0,T],\\
 \mathbf{u}=0&\text{ on }\  \partial\Omega\times(0, T),\\
\mathbf{u}(y,0)=0&\text{ in }\  \Omega,\\
u_i\geq 0&\text{ in }\  {\bar{\Omega}\times[0,T]}.
\end{cases}
\end{equation}
    is equivalent to deriving the uniqueness of $\textbf{p}$ in 
    \begin{equation}\label{TimeMainPb}
\begin{cases}
{_{0}D}_t^{\bm{\beta}}\mathbf{u}-\Delta_x \mathbf{u}=\mathbf{p}(x,t)\mathbf{u}+\mathbf{F}(x,t,\mathbf{u})+\mathbf{q} &\text{ in }\  \Omega\times(0,T],\\
 \mathbf{u}=0&\text{ on }\  \partial\Omega\times(0, T),\\
\mathbf{u}(x,0)=0&\text{ in }\  \Omega,\\
u_i\geq 0&\text{ in }\  {\bar{\Omega}\times[0,T]},
\end{cases}
\end{equation}
    when considering the source-to-boundary map $\mathcal{M}$ mapping $\textbf{p}\text{ or }\tilde{\textbf{p}}$ to the Neumann boundary of the elliptic space operator acting on $\textbf{u}$, i.e. either $\mathbf{d}\Delta_y +\bm{\alpha}\cdot\nabla_y $ or $\Delta_x$, given input $\textbf{F}$ and $\textbf{q}$.
	 
\end{lemma}

\begin{proof}
    We use the transformation $y:=\frac{x}{\sqrt{d_i}}$, $u_i(x,t):=u_i(y,t)e^{-\frac{\alpha_ix}{2\sqrt{d_i}}}$, $p_i:=\tilde{p}_i+\frac{\alpha_i^2}{4d_i}$. Then, 
\begin{align*}
    &\quad{_{0}D}_t^{\bm{\beta}}u_i(y,t)-d_i\Delta_y u_i(y,t)+\alpha_i\cdot\nabla_y u_i(y,t)-\tilde{p}_i(y,t)u_i(y,t)-\tilde{F}_i(y,t,u_i(y,t))-q_i\\
    &=e^{\frac{\alpha_ix}{2\sqrt{d_i}}}{_{0}D}_t^{\bm{\beta}}u_i(x,t)-\nabla_x\cdot\nabla_x \left(e^{\frac{\alpha_ix}{2\sqrt{d_i}}} u_i(x,t)\right)+\frac{\alpha_i}{\sqrt{d_i}}\nabla_x \cdot\left(e^{\frac{\alpha_ix}{2\sqrt{d_i}}}u_i(x,t)\right)\\&\quad-\tilde{p}_i(y,t)e^{\frac{\alpha_ix}{2\sqrt{d_i}}}u_i(x,t)-\tilde{F}_i(y,t,e^{\frac{\alpha_ix}{2\sqrt{d_i}}}u_i(x,t))-q_i\\
    &=e^{\frac{\alpha_ix}{2\sqrt{d_i}}}{_{0}D}_t^{\bm{\beta}}u_i(x,t)-\nabla_x\cdot \left(\frac{\alpha_i}{2\sqrt{d_i}}e^{\frac{\alpha_ix}{2\sqrt{d_i}}} u_i(x,t)+e^{\frac{\alpha_ix}{2\sqrt{d_i}}} \nabla_x u_i(x,t)\right)+\frac{\alpha_i}{\sqrt{d_i}}\nabla_x \cdot\left(e^{\frac{\alpha_ix}{2\sqrt{d_i}}}u_i(x,t)\right)\\&\quad-\tilde{p}_i(y,t)e^{\frac{\alpha_ix}{2\sqrt{d_i}}}u_i(x,t)-\tilde{F}_i(y,t,e^{\frac{\alpha_ix}{2\sqrt{d_i}}}u_i(x,t))-q_i\\
    &=e^{\frac{\alpha_ix}{2\sqrt{d_i}}}{_{0}D}_t^{\bm{\beta}}u_i(x,t)-\frac{\alpha_i}{2\sqrt{d_i}}\nabla_x\cdot\left(e^{\frac{\alpha_ix}{2\sqrt{d_i}}} u_i(x,t)\right)-\frac{\alpha_i}{2\sqrt{d_i}}e^{\frac{\alpha_ix}{2\sqrt{d_i}}} \cdot\nabla_x u_i(x,t)-e^{\frac{\alpha_ix}{2\sqrt{d_i}}} \Delta_x u_i(x,t)\\&\quad+\frac{\alpha_i}{\sqrt{d_i}}\nabla_x \cdot\left(e^{\frac{\alpha_ix}{2\sqrt{d_i}}}u_i(x,t)\right)-\tilde{p}_i(y,t)e^{\frac{\alpha_ix}{2\sqrt{d_i}}}u_i(x,t)-\tilde{F}_i(y,t,e^{\frac{\alpha_ix}{2\sqrt{d_i}}}u_i(x,t))-q_i\\
    &=e^{\frac{\alpha_ix}{2\sqrt{d_i}}}{_{0}D}_t^{\bm{\beta}}u_i(x,t)-\frac{\alpha_i}{2\sqrt{d_i}}e^{\frac{\alpha_ix}{2\sqrt{d_i}}} \cdot \nabla_x u_i(x,t)-e^{\frac{\alpha_ix}{2\sqrt{d_i}}} \Delta_x u_i(x,t)\\&\quad+\frac{\alpha_i}{2\sqrt{d_i}}\nabla_x \cdot\left(e^{\frac{\alpha_ix}{2\sqrt{d_i}}}u_i(x,t)\right)-\tilde{p}_i(y,t)e^{\frac{\alpha_ix}{2\sqrt{d_i}}}u_i(x,t)-\tilde{F}_i(y,t,e^{\frac{\alpha_ix}{2\sqrt{d_i}}}u_i(x,t))-q_i\\
    &=e^{\frac{\alpha_ix}{2\sqrt{d_i}}}{_{0}D}_t^{\bm{\beta}}u_i(x,t)-\frac{\alpha_i}{2\sqrt{d_i}}e^{\frac{\alpha_ix}{2\sqrt{d_i}}} \cdot\nabla_x u_i(x,t)-e^{\frac{\alpha_ix}{2\sqrt{d_i}}} \Delta_x u_i(x,t)\\&\quad+\frac{\alpha_i^2}{4d_i}e^{\frac{\alpha_ix}{2\sqrt{d_i}}}u_i(x,t)+\frac{\alpha_i}{2\sqrt{d_i}} e^{\frac{\alpha_ix}{2\sqrt{d_i}}}\cdot\nabla_x u_i(x,t)-\tilde{p}_i(y,t)e^{\frac{\alpha_ix}{2\sqrt{d_i}}}u_i(x,t)-\tilde{F}_i(y,t,e^{\frac{\alpha_ix}{2\sqrt{d_i}}}u_i(x,t))-q_i\\
    &=e^{\frac{\alpha_ix}{2\sqrt{d_i}}}{_{0}D}_t^{\bm{\beta}}u_i(x,t)-e^{\frac{\alpha_ix}{2\sqrt{d_i}}} \Delta_x u_i(x,t)+\left(\frac{\alpha_i^2}{4d_i}-\tilde{p}_i(y,t)\right)e^{\frac{\alpha_ix}{2\sqrt{d_i}}}u_i(x,t)-\tilde{F}_i(y,t,e^{\frac{\alpha_ix}{2\sqrt{d_i}}}u_i(x,t))-q_i.
\end{align*}
Since $\alpha_i$ and $d_i$ are known and fixed, the two systems are equivalent, by setting $F_i(x,t,u_i)=\tilde{F}_i(y,t,e^{\frac{\alpha_ix}{2\sqrt{d_i}}}u_i(x,t))$, and recovering $p_i$ is the same as recovering $\tilde{p}_i$, while the recovery of $\bm{\beta}$ remains unchanged.
\end{proof}

Therefore, it suffices to consider the unique identifiability issue of $\textbf{p}$ and $\mathbf{F}$ in \eqref{TimeMainPb}, given by the measurement map
\begin{equation}\label{MeasureMapTime}
        \Lambda^2_{\textbf{p},\mathbf{F}}(\textbf{q})= \partial_\nu \textbf{u}|_{\Gamma\times(0,T)},
	\end{equation}
where $\nu(x)$ is the unit outward normal vector on the boundary and $u$ is the solution to \eqref{TimeMainPb}, defined once again in the distributional sense for any $h$ defined on the boundary $\Gamma\subset\partial\Omega$.

For the nonlocal time-fractional diffusion systems, we choose the input function $q$ such that 
\begin{equation}
    q(x,t;\epsilon)=\sum_{l=1}^K\epsilon_lg_l(x,t),\quad (x,t)\in \Omega\times(0,T],
\end{equation}
where %we consider the same $g_1$ given by \[g_1(x,t)=\varphi^n(x)V(t)>0, \quad \varphi(x) \text{ is a complete set in }L^2(\Omega),\]
%while 
\begin{equation}\label{g1}
    g_1(x,t)=\phi^n(x,t)>0, \quad \phi(x,t) \text{ is a complete set in }L^2(\Omega\times(0,T)),
\end{equation}
\begin{equation}\label{g2}
    g_2(x,t)=e^{-at}f(x)>0, \quad a>0,
\end{equation}
and $g_l>0$ for $l=3,\dots,K$.

\subsection{Determining the Potentials $\mathbf{p}$}

Then, the first order linearised system is given by 
\begin{equation}\label{TimeLinear1u1}
\begin{cases}
{_{0}D}_t^{\bm{\beta}}u^{(1)}-\Delta u^{(1)}=pu^{(1)}+\phi^n(x,t)&\text{ in }\  \Omega\times(0,T],\\
u^{(1)}=0&\text{ on }\  \partial\Omega\times(0, T),\\
u^{(1)}(x,0)=0&\text{ in }\  \Omega,
\end{cases}
\end{equation}

Set $u_k$ to be the solutions of \eqref{TimeLinear1u1} corresponding to the input functions $\phi^n$ and $p^k$, $k=1,2$. 
Observe that, unlike the space-fractional case where we consider an accessible region $\Omega_a$, we consider partial boundary data $\Gamma\subset\partial\Omega$ in the time-fractional case. Therefore, for a given nonzero nonnegative function $h\in C_c(\Gamma\times(0,T))$, we set $h_0=h$ on $\Gamma\times(0,T)$ and $h_0=0$ on $(\partial\Omega\setminus\Gamma)\times(0,T)$, and introduce the function $ w(x,t;p)$ as the solution of the following adjoint problem
\begin{equation}\label{AdjointPbTime}
\begin{cases}
{_{t}D}_T^{\bm{\beta}}w-\Delta w=pw &\text{ in }\  \Omega\times(0,T),\\
 w=h_0 &\text{ on }\  \partial\Omega\times(0, T),\\
w(x,T)=0, &\text{ in }\  \Omega,\\
\end{cases}
\end{equation}
Then, by Theorem 7.9 in \cite{JinbookTimeFrac}, it is known that $w$ admits H\"older regularity, and is therefore uniformly bounded for a bounded domain $\Omega\times(0,T)$. 

Multiplying \eqref{TimeLinear1u1} by $w$, by Green's second identity, we compute
\begin{align*}
&\int_{0}^T \int_{\Omega}\phi^n(x,t)w(x,t;p^1)\,dxdt\\
=&\int_{0}^T\int_{\Omega}\left({_{0}D}_t^{\bm{\beta}} u_1- \Delta u_1-p^1u_1\right)w(x,t;p^1)\,dxdt\\
=&\int_{0}^T\int_{\Omega}\left({_{t}D}_T^{\bm{\beta}} w -\Delta w-p^1w\right)u_1\,dxdt -\int_0^T\int_{\partial\Omega}w\nabla u_1\cdot\nu\,dxdt\\
=&-\int_0^T\int_\Gamma h\partial_\nu u_1 \,dxdt,
\end{align*}
by the construction of $w$ in \eqref{AdjointPbTime}. 
Here, we recall that $\nu(x)$ is the unit outward normal vector on the boundary, 
and the second equality uses the following fractional integration by parts formula (see Lemma 2.1 of \cite{TingWang2016CaputoByPartsACFuncs}):
\begin{equation*}
\int_0^T{_{0}D_t^{\bm{\beta}}}u(t)w(t)dt=\int_0^T u(t){_{t}D_T^{\bm{\beta}}}w(t)dt,\quad\text{ for any functions }\ u,w\in AC[0,T], \text{and}\ w(T)=0,
\end{equation*}
where $AC[0,T]$ is the space of absolutely continuous functions on $[0,T]$.
Similarly, for the same equation for $u_2$ with potential $p^2$, we have 
\begin{equation*}
\int_{0}^T \int_{\Omega}\phi^n(x,t)w(x,t;p^2)dxdt=-\int_0^T\int_\Gamma h\partial_\nu u_2 \,dxdt.
\end{equation*}
By \eqref{condtionTime}, we obtain the equality 
\[\int_{0}^T \int_{\Omega}\phi^n(x,t)w(x,t;p^1)\,dxdt=\int_{0}^T \int_{\Omega}\phi^n(x,t)w(x,t;p^2)\,dxdt.\]
The completeness of $\{\phi^n\}_n$ then implies 
\begin{equation}w(x,t;p^1)=w(x,t;p^2)=:w(x,t)\quad \text{ in }\Omega\times(0,T).\end{equation}

Substituting this into \eqref{AdjointPbTime} and taking the difference of the two expressions corresponding to $p^1$ and $p^2$, we find
\begin{equation*}
(p^1-p^2)w(x,t)=0 \quad \text{ in }\Omega\times(0,T).
\end{equation*}
The maximum principle for time-fractional diffusion equation (Lemma \ref{StrongMaxPrincipleTime}) can then be applied to deduce that $w(x,t)>0$, and we have the result $p^1(x,t)=p^2(x,t)$ in $\Omega\times(0,T)$.

\begin{remark}\label{MoreDataPxtRemark}
    Note that if we consider instead $g_1=\varphi^n(x)V(t)$ as in the proof of the space-fractional case, we can obtain $p^1(x)=p^2(x)$ for $p$ independent of time, by considering $V=v$ and $V={_{0}D_t^{\bm{\beta}}v}$. This means that with two data points in time, we can obtain the result when $p$ is independent of the time dimension. Instead, here we consider $g_1=\phi^n(x,t)$ dense, which means we have to input more source functions $q$, but we obtain the result for a more general $p(x,t)$.
\end{remark}

\subsection{Determining the Time-Fractional Orders \texorpdfstring{$\bm{\beta}$}{\beta}}
Next, to obtain the unique determination of the fractional orders, we make use of the measurement $\Lambda^3$ and the first-order linearisation about $g_2$, for $g_2$ given in \eqref{g2}. 
Then, we have
\begin{equation}\label{linear21}
\begin{cases}
{_{0}D}_t^{\bm{\beta}}\mathbf{u}^{(2)}-\Delta \mathbf{u}^{(2)}=\mathbf{p}\mathbf{u}^{(2)}+e^{-at}\mathbf{f}(x)&\text{ in }\  \Omega\times(0,T],\\
\mathbf{u}^{(2)}=0&\text{ on }\  \partial\Omega\times(0, T),\\
\mathbf{u}^{(2)}(x,0)=0&\text{ in }\  \Omega,\\
\mathbf{u}\geq0&\text{ in }\  \bar{\Omega}\times[0,T],
\end{cases}
\end{equation}
where $e^{-at}\mathbf{f}(x)=e^{-at}f(x)(1,1,\dots,1)$ for $a,f>0$.

Observe that \eqref{linear21} is now uncoupled, and we can consider the equation for each $u_i$ individually. Then, $u_i$ is known to be time-analytic, by the time-analyticity of the associated Green's function and the time-analyticity of the source function $e^{-at}f(x)$ (see, for instance, Theorem 3.7(i) of \cite{JinbookTimeFrac} or the proof of Lemma 5(ii) of \cite{LiHuangLiu2024-TimeFracCoupledDiffusion-InversePbIdentifyExponent}). Hence, we can uniquely extend the solution $\mathbf{u}$ from $(0,T)$ to $(0,\infty)$, allowing us to apply the Laplace transform.  %Wright function https://link.springer.com/content/pdf/10.1007/BF02832308.pdf

Consider two fractional orders $\bm{\beta}^1$ and $\bm{\beta}^2$, with corresponding solutions $\mathbf{u}_1$ and $\mathbf{u}_2$. 
Applying the Laplace transform on both sides of the equation \eqref{linear21} for $\bm{\beta}^k$ ($k=1,2$), for $s>0$, we find

\begin{equation}
\begin{cases}
(-\Delta-\mathbf{p}+s^{\bm{\beta}^1})\hat{\mathbf{u}}_1=1/(s+a){f(x)} &\text{ in }\  \Omega,\\
\hat{\mathbf{u}}_1=0 &\text{ on }\  \partial\Omega,
\end{cases}
\end{equation}
and
\begin{equation}\label{u2}
\begin{cases}
(-\Delta-\mathbf{p}+s^{\bm{\beta}^2})\hat{\mathbf{u}}_2=1/(s+a){f(x)} &\text{ in }\  \Omega,\\
\hat{\mathbf{u}}_2=0&\text{ on }\  \partial\Omega,
\end{cases}
\end{equation}
where the $i$-th component of $\hat{\mathbf{u}}_k$ is the Laplace transform of the $i$-th component of $\mathbf{u}_k^{(2)}$. Here, $s^{\bm{\beta}}$ corresponding to the $i$-th equation is given by $\sum_{j=1}^{B_i}b_{j,i}s^{\beta_{j,i}}$. 

Taking the difference between the above two equations, it turns out that $\hat{\mathbf{u}}_1-\hat{\mathbf{u}}_2$
satisfies
\begin{equation}\label{linear2}
\begin{cases}
(-\Delta-\mathbf{p}+s^{\bm{\beta}^1})(\hat{\mathbf{u}}_1-\hat{\mathbf{u}}_2)(s)=(s^{\bm{\beta}^2}-s^{\bm{\beta}^1})\hat{\mathbf{u}}_2 &\text{ in }\  \Omega,\\
 \hat{\mathbf{u}}_1-\hat{\mathbf{u}}_2=0 &\text{ on }\  \partial\Omega.
\end{cases}
\end{equation}

We want to show that $\bm{\beta}^1=\bm{\beta}^2$, given \eqref{condtionTimeOrder}. Assume on the contrary that $\bm{\beta}^1\neq\bm{\beta}^2$, i.e. 
\[\begin{pmatrix}
    b_{1,1}(x)s^{\beta_{1,1}^1}+\cdots+b_{B_1,1}(x)s^{\beta_{B_1,1}^1}\\\vdots\\b_{1,M}(x)s^{\beta_{1,M}^1}+\cdots+b_{B_M,M}(x)s^{\beta_{B_M,M}^1}
\end{pmatrix}\neq 
\begin{pmatrix}
    b_{1,1}(x)s^{\beta_{1,1}^2}+\cdots+b_{B_1,1}(x)s^{\beta_{B_1,1}^2}\\\vdots\\b_{1,M}(x)s^{\beta_{1,M}^2}+\cdots+b_{B_M,M}(x)s^{\beta_{B_M,M}^2}
\end{pmatrix},\]
where $\beta_{j_1,i}^k<\beta_{j_2,i}^k$ when $j_1<j_2$ for each $i,k$. Without loss of generality, and renumbering the $u_i$'s if necessary, we order all $\beta_{j,i}$ for each $k$, and suppose that $\hat{\beta}$ is the smallest $(j,i)$-th order $\beta_{j,i}$ such that $\hat{\beta}^1< \hat{\beta}^2$. 

We introduce the auxiliary function $w$ defined by
\begin{equation}\label{w}
w(s)=\frac{(\hat{\mathbf{u}}_1-\hat{\mathbf{u}}_2)}{s^{\hat{\beta}^1}-s^{\hat{\beta}^2}}, s>0,
\end{equation}
Multiplying both sides of \eqref{linear2} by $1/(s^{\hat{\beta}^1}-s^{\hat{\beta}^2})$, we see that $w$ satisfies the boundary value problem
\begin{equation}\label{wEq}
\begin{cases}
(-\Delta-\mathbf{p}+s^{\bm{\beta}^1})(\mathbf{w})(s)=\Sigma  \hat{\mathbf{u}}_2 &\text{ in }\  \Omega,\\
w=0 &\text{ on }\  \partial\Omega,
\end{cases}
\end{equation}
where \[\Sigma=\begin{pmatrix}
    b_{1,1}(x)\cfrac{s^{\beta_{1,1}^2}-s^{\beta_{1,1}^1}}{s^{\hat{\beta}^1}-s^{\hat{\beta}^2}}+\cdots+b_{B_1,1}(x)\cfrac{s^{\beta_{B_1,1}^2}-s^{\beta_{B_1,1}^1}}{s^{\hat{\beta}^1}-s^{\hat{\beta}^2}}\\\vdots\\ -b_{j,i}+b_{j+1,i}\cfrac{s^{\beta_{j+1,i}^2}-s^{\beta_{j+1,i}^1}}{s^{\hat{\beta}^1}-s^{\hat{\beta}^2}}+\cdots+b_{B_i,i}(x)\cfrac{s^{\beta_{B_i,i}^2}-s^{\beta_{B_i,i}^1}}{s^{\hat{\beta}^1}-s^{\hat{\beta}^2}}\\\vdots\\b_{1,M}(x)\cfrac{s^{\beta_{1,M}^2}-s^{\beta_{1,M}^1}}{s^{\hat{\beta}^1}-s^{\hat{\beta}^2}}+\cdots+b_{B_M,M}(x)\cfrac{s^{\beta_{B_M,M}^2}-s^{\beta_{B_M,M}^1}}{s^{\hat{\beta}^1}-s^{\hat{\beta}^2}},
\end{pmatrix}\]
with $s^{\beta_{j',i'}^1}-s^{\beta_{j',i'}^2}=0$ for $\beta_{j',i'}<\beta_{j,i}$. 
From the property of the Laplace transform, it follows that $\hat{\mathbf{u}}_1(s)$ and $\hat{\mathbf{u}}_2(s)$ are analytic with respect to $s>0$, so, by definition, $\mathbf{w}(s)$ is continuous for $s>0$. 

Next we discuss the limit of $\mathbf{w}(s)$ as $s\rightarrow 0$. We claim that
\begin{equation}\label{w0}
\lim_{s\rightarrow 0}\mathbf{w}(s)=\mathbf{w_0}\ \mathrm{in}\ [H^2(\Omega)]^M,
\end{equation}
where $\mathbf{w_0}$ solves the following boundary value problem:
\begin{equation}\label{w0Eq}
\begin{cases}
(-\Delta-\mathbf{p})\mathbf{w_0}=D \hat{\mathbf{u}}_2 &\text{ in }\  \Omega,\\
\mathbf{w_0}=0 &\text{ on }\  \partial\Omega,
\end{cases}
\end{equation}
Here $D$ is a matrix with
\[D=\begin{pmatrix}
    b_{1,1}(x)\gamma_{1,1}+\cdots+b_{B_1,1}(x)\gamma_{B_1,1}\\\vdots\\ -b_{j,i}+b_{j+1,i}\gamma_{j+1,i}+\cdots+b_{B_i,i}(x)\gamma_{B_i,i}\\\vdots\\b_{1,M}(x)\gamma_{1,M}+\cdots+b_{B_M,M}(x)\gamma_{B_M,M},
\end{pmatrix}\]
where $\gamma_{j',i'}=0$ if $\beta_{j',i'}<\beta_{j,i}$ and $-1$ otherwise.
% By noting the definition of the diagonal matrix $D$, we find that the $k$-th entry of the diagonal matrix $\Sigma-D$ in read
% \begin{equation}\label{linear2}
% \begin{cases}
% \frac{s^{\alpha_k}-s^{\beta_k}}{s^{\alpha_{k_1}}-s^{\beta_{k_1}}}-1=
% \frac{s^{\beta_{k_1}}-s^{\beta_k}}{s^{\alpha_{k_1}}-s^{\beta_{k_1}}}=
% \frac{s^{\beta_{k_1}-\alpha_{k_1}}-s^{\beta_k-\alpha_{k_1}}}{1-s^{\beta_{k_1}-\alpha_{k_1}}}, \ k\in I_1\\
% \frac{s^{\alpha_k}-s^{\beta_k}}{s^{\alpha_{k_1}}-s^{\beta_{k_1}}}=
% \frac{s^{\alpha_k-\alpha_{k_1}}-s^{\beta_k-\alpha_{k_1}}}{1-s^{\beta_{k_1}-\alpha_{k_1}}}, \ k\in I_2\\
% 0,\ \   k\geq k_1
% \end{cases}
% \end{equation}
Setting $\mathbf{z}(s)=\mathbf{w}(s)-\mathbf{w_0}(s)$, we take the difference between \eqref{wEq} and \eqref{w0Eq} to obtain
\begin{equation}
\begin{cases}
(-\Delta-\mathbf{p}+s^{\bm{\beta}^1})(\mathbf{z})(s)=(\Sigma-D)\hat{\mathbf{u}}_2 -s^{\bm{\beta}^1} \mathbf{w}_0&\text{ in }\  \Omega,\\
\mathbf{z}=0 &\text{ on }\  \partial\Omega.
\end{cases}
\end{equation}
Since these equations are uncoupled, this is a second-order elliptic equation in $z_{i'}$ for each $i'$, and the regularity estimate for elliptic equations implies that
\begin{equation}
\norm{z_{i'}(s)}_{H^2(\Omega)}\leq C\left(\norm{(\Sigma-D)_{i'}\hat{u}_{2,i'}}_{L^2(\Omega)}+\|s^{\bm{\beta}^1_{i'}} w_{0,i'}\|_{L^2(\Omega)}\right).
\end{equation}
But $\hat{u}_{2,i'}$ and $w_{0,i'}$ satisfy their own elliptic equations \eqref{u2} and \eqref{w0Eq}, respectively. Hence, 
\begin{equation}\|s^{\bm{\beta}^1_{i'}} w_{0,i'}\|_{L^2(\Omega)}\leq C_1 s^{\bm{\beta}^1}\|\hat{u}_{2,i}\|_{L^2(\Omega)}\sum_{j'=1}^{B_{i'}}|b_{j',i'}|
\end{equation}
and 
\begin{equation}\label{DiffEq}
\norm{(\Sigma-D)_{i'}\hat{u}_{2,i'}}_{L^2(\Omega)}\leq 
\norm{\hat{u}_{2,i'}}_{L^2(\Omega)}\left[b_{1,i'}(x)\left(\cfrac{s^{\beta_{1,i'}^2}-s^{\beta_{1,i'}^1}}{s^{\hat{\beta}^1}-s^{\hat{\beta}^2}}-\gamma_{1,i'}\right)+\cdots+b_{B_{i'},i'}(x)\left(\cfrac{s^{\beta_{B_{i'},i'}^2}-s^{\beta_{B_{i'},i'}^1}}{s^{\hat{\beta}^1}-s^{\hat{\beta}^2}}-\gamma_{B_{i'},i'}\right)\right],
\end{equation}
% \[\bigg{(}\sum_{k\in I_1}\frac{s^{\beta_{k_1}-\alpha_{k_1}}-s^{\beta_k-\alpha_{k_1}}}{(1-s^{\beta_{k_1}-\alpha_{k_1}})(s+a)}
% +\sum_{k\in I_2}\frac{s^{\alpha_k-\alpha_{k_1}}-s^{\beta_k-\alpha_{k_1}}}{(1-s^{\beta_{k_1}-\alpha_{k_1}})(s+a)}\bigg{)}\]
where 
\begin{equation}
\norm{\hat{u}_{2,i'}}_{L^2(\Omega)}\leq \frac{C_2}{s+a}\norm{f}_{L^2(\Omega)}.
\end{equation}
Now we investigate the terms in \eqref{DiffEq}. We observe that 
\[\cfrac{s^{\beta_{j',i'}^2}-s^{\beta_{j',i'}^1}}{s^{\hat{\beta}^1}-s^{\hat{\beta}^2}}-\gamma_{j',i'}=\begin{cases}
    0-0=0& \text{ if }\beta_{j',i'}^1,\beta_{j',i'}^2<\hat{\beta}^1,\\
    \cfrac{s^{\beta_{j',i'}^2}-s^{\beta_{j',i'}^1}}{s^{\hat{\beta}^1}-s^{\hat{\beta}^2}}+1=s^{\hat{\beta}^1}\cfrac{s^{\beta_{j',i'}^2-\hat{\beta}^1}-s^{\beta_{j',i'}^1-\hat{\beta}^1}+1-s^{\hat{\beta}^2-\hat{\beta}^1}}{1-s^{\hat{\beta}^2-\hat{\beta}^1}}& \text{ if }\hat{\beta}^1<\beta_{j',i'}^1,\beta_{j',i'}^2,\\
    \cfrac{s^{\hat{\beta}^2}-s^{\hat{\beta}^1}}{s^{\hat{\beta}^1}-s^{\hat{\beta}^2}}+1=0 & \text{ if }i=i',j=j'.
\end{cases}\]
When $s\to0$, 
\[s^{\hat{\beta}^1}\cfrac{s^{\beta_{j',i'}^2-\hat{\beta}^1}-s^{\beta_{j',i'}^1-\hat{\beta}^1}+1-s^{\hat{\beta}^2-\hat{\beta}^1}}{1-s^{\hat{\beta}^2-\hat{\beta}^1}}\to0,\] so we have that $\Sigma-D\to0$ as $s\to0$.
Combining the above estimates, we obtain that $\lim\limits_{s\rightarrow 0}\|z_i(s)\|_{L^2(\Omega)}=0$.

Next, applying the maximum principle on \eqref{linear21}, we have that $\mathbf{u}^{(2)}_2>0$, hence its Laplace transform $\hat{\mathbf{u}}_2>0$. Consequently, by \eqref{w0Eq}, since $D$ is negative, $\mathbf{w}_0<0$ in $\Omega$, that is, $w_{0,i'}<0$ in $\Omega$ for all $i'=1,\dots,M$. Then, at the observation point $x_0\in \Omega$, from the relation $\mathbf{w}_0=\lim_{s\rightarrow 0} \mathbf{w}(s)$, we have
 \begin{equation}
\lim_{s\rightarrow 0}w_{i'}(x_0;s)=w_{0,i'}(x_0)<0.
\end{equation}
This indicates that we can choose a small $0<\epsilon<1$ such that $w_{i'}(x_0;s)<0$ for any $s\in(0,\epsilon)$ and $i'$, which implies that
 \begin{equation}
u_{i',1}(x_0;s)<u_{i',2}(x_0;s) \ \  \text{ for all }\  s\in (0,\varepsilon), i'
\end{equation}
This yields a contradiction since $u_{i_0,1}=u_{i_0,2}$ at $x_0\times (0,T)$ implies $u_{i_0,1}=u_{i_0,2}$ at $\{x_0\}\times(0,\infty)$ by its time analyticity, and hence $\hat{u}_{i_0,1}(x_0;s)=\hat{u}_{i_0,1}(x_0;s)$ for any $s>0$.

\begin{remark}
    Our result differs from \cite{LiHuangLiu2024-TimeFracCoupledDiffusion-InversePbIdentifyExponent} in many different aspects, including but not limited to
    \begin{enumerate}
        \item We considered a system of equations coupled in the interaction function term $\mathbf{F}$ instead of in the time-fractional term;
        \item We considered a time-fractional term ${_{0}D}_t^{\bm{\beta}}$ consisting of multiple time-fractional orders $\beta_{\cdot,j}$ with variable coefficients $b_{j,\cdot}(x)$;
        \item We considered a different measurement map, given by the source-to-boundary measurement.
    \end{enumerate}
    Hence, we made use of the input source function $\mathbf{g}_2$ to recover the time-fractional orders. This method is technically different from that in \cite{LiHuangLiu2024-TimeFracCoupledDiffusion-InversePbIdentifyExponent}.
\end{remark}

\subsection{Determining the Interaction Functions $\mathbf{F}$}\label{subsect:FTime}

Finally, we recover the interaction functions $\mathbf{F}$. We make use of the higher-order linearised systems, and the proof is similar to that of Section \ref{sect:ProofSourceSpace}, and we show the main differences.

As an example, we once again consider the equation for $u_{1,k}$, dropping all subscripts `1' as before. The second-order linearisation is given, for $k=1,2$, by
\begin{equation}\label{linear2time}
\begin{cases}
{_{0}D}_t^{\bm{\beta}}u^{(1,2)}_k-\Delta u^{(1,2)}_k=p(x,t)u^{(1,2)}_k+2F^{(20\cdots 0)}_ku^{(1)}u^{(2)}\\
\qquad\qquad\qquad\qquad\qquad\qquad\qquad+\sum\limits_{j=2}^KF^{(10\cdots1\cdots0)}_k(u^{(1)}u_j^{(2)}+u^{(2)}u_j^{(1)}) &\text{ in }\  \Omega\times(0,T],\\
u^{(1,2)}_k=0&\text{ on }\  \partial\Omega\times(0, T),\\
u^{(1,2)}_k(x,0)=0&\text{ in }\ \Omega,\\
u_k\geq0&\text{ in }\ \bar{\Omega}\times[0,T],
\end{cases}
\end{equation}
since $\mathbf{p}$ have already been determined.

Again, we note that in the right-hand-side terms, the Taylor coefficients of the interaction function $\mathbf{F}$ are only correlated with the lower-order first-order linearised solutions $u^{(1)}$ and $u^{(2)}$, for which the maximum principle (Lemma \ref{StrongMaxPrincipleTime} holds. Taking the difference of the two equations for $k=1,2$, multiplying both sides of the equation by the $w$ in the time-adjoint problem \eqref{AdjointPbTime}, and integrating by parts, we once again obtain the same integral equality \eqref{SpaceEq}:
\begin{multline}2(F^{(20\cdots 0)}_1-F^{(20\cdots 0)}_2)\int_0^T\int_\Omega u^{(1)}u^{(2)}w\,dxdt\\+\sum\limits_{j=2}^K(F^{(10\cdots1\cdots0)}_1-F^{(10\cdots1\cdots0)}_2)\int_0^T\int_\Omega (u^{(1)}u_j^{(2)}+u^{(2)}u_j^{(1)})w\,dxdt=0.\end{multline} 
Then, when all except one of the $F^{(\cdot)}$'s are fixed and known, say $\hat{F}=F^{(20\cdots0)}$ or $F^{(10\cdots1\cdots0)}$, 
by the maximum principle on $u^{(1)}, m^{(2)}, u^{(2)}, m^{(1)} >0$ in $\Omega\times(0,T]$, we have $\hat{F}_1=\hat{F}_2$. This similarly holds for all second-order Taylor coefficients of $\mathbf{F}$, and by mathematical induction, this is true for the $l$-th order Taylor coefficient of $F$, $N\geq 1$. Hence, $\mathbf{F}_1=\mathbf{F}_2$ in $\Omega\times (0,T]$ and the result is proved.

\begin{remark}
    It may be possible to improve this result to consider $F$ non-constant, as in Remark \ref{rmk:SpaceExtNonConstant}. However, this requires a Runge-type density result, which is not yet known for time-fractional diffusion equations. The only known unique continuation results are in \cite{LinNakamura2019MultiTermTimeFracUCP,LinNakamura2021MultiTermTimeFracUCP}.
\end{remark}

\section{Applications}\label{sec:apply}

Fractional calculus and fractional differential equations have been more and more extensively used in various scientific fields, including physics \cite{Metzler20001,UchaikinI,UchaikinII}, chemistry \cite{FracChemistry1,FracChemistry2,FracChemistry3,FracChemistry4}, biology \cite{HenryLanglandsWearne20011FracCableNeurons}, ecology \cite{Shi2024DCDSAdvectiveKellerSegel,HeShouWu2024ChemotaxisFrac}, engineering \cite{SamkoKilbas,Kilbas,BrockmannHufnagelGeisel2006NatureAnomalousTravel}, medicine \cite{HallBarrick2012FracCable}, hydrology \cite{SchumerBensonMeerschaertBaeumer2003FracHydrology,ChakrabortyMeerschaertLim2009FracHydrology,Momani2006HydrologyBurgers,FracAdvectionDispersionExample,meerschaert2006fractional}, finance \cite{ScalasGorenfloMainardi2000FracFinance}, and so on. 
In this section, we discuss several applications of our results in different settings, with a focus on biological applications.

\subsection{Epidemics Models Based on the Nonlocal Nonlinear System}

Throughout history, there have been numerous global epidemics of deadly diseases. Although advancements in hygiene, healthcare accessibility, and medicine have reduced mortality and morbidity in recent decades, challenges remain. Mathematical models play a crucial role in understanding and examining disease dynamics, through the modelling of interactions among hosts, pathogens, and vector. These models help forecast disease spread and estimate infection numbers. 

In the study of the spread of infectious diseases, incorporating space-fractional operators into epidemic models allows for a more realistic representation of disease transmission dynamics in spatially heterogeneous environments. These models can capture the spatial spread of the disease, considering factors such as population density, movement patterns, and interactions between individuals. By incorporating fractional derivatives, these models can account for non-local and long-range interactions, which are particularly relevant in the context of disease transmission. Such versatile models have been applied to study a wide range of epidemics, including HIV \cite{FracHIV}, malaria \cite{FracMalaria}, COVID \cite{FracCOVID1,FracCOVID2,FracCOVID3,FracCOVID4}, measles \cite{FracMeasles}, typhoid \cite{FracTyphoid}, and so on (see for instance \cite{FracSmoking,FracEpidemicModelsReview} and the references therein). By incorporating fractional calculus into mathematical models, researchers can gain deeper insights into the complex behaviours of epidemic systems, address key challenges in biomedical research and healthcare, obtain a better understanding of disease transmission dynamics, and improve medical treatment assessments. 

Denoting $u_i$ for each type of population, the nonlocal susceptible-infected-recovered (SIR)-type model with anomalous diffusion in a heterogeneous environment is given by the following form:
\begin{equation}\label{EpidemicModel}
\partial_t \textbf{u}-\mathcal{L} \textbf{u}+\bm{\alpha}\cdot\nabla\mathbf{u}=\mathbf{p}(x)\textbf{u}+\textbf{F}(x,t,\textbf{u}) \quad\text{ in }\  \Omega\times(0,T],
\end{equation}
where $\mathbf{p}$ denotes the death or recovery rates and $\mathbf{F}$ denotes the infection model. Some examples of infection models \cite{FracBioPop} include the linear Malthusian Law where $\mathbf{F}(\mathbf{u})=\sum\limits_i c_iu_i$ and the quadratic Verhulst Law where $\mathbf{F}(\mathbf{u})=\sum\limits_i\sum\limits_j c_{ij}u_iu_j$ with the possibility of $i=j$. The infection rate may also depend further on the size of the different populations $u_i$, leading to higher order terms and nonlinearity for $\mathbf{F}$ \cite{SIRModel1,FracBioPopNonlinear,HsuYang2013CoupledEpidemic}.

Here, the advection term $\bm{\alpha}\cdot\nabla\mathbf{u}$ indicates that individuals can undergo passive movement in a specific direction \cite{HuHuoYuan2024DCDSAdvectiveSIRModel}. For instance, infected aquatic creatures may be carried along by water flow \cite{WaterAdvection1,WaterAdvection2}, while infected birds migrate seasonally and as a result of winds \cite{WindAdvection}. Such behaviour can be characterised by the inclusion of this advection term.

Another important point to note is that, in this model, there may be more than 3 species (the susceptibles $S=u_1$, infected $I=u_2$, and recovered $R=u_3$). For instance, one may consider a non-compliant part of the population ($S'=u_4,I'=u_5,R'=u_6$) which do not comply with the preventive measures that have been implemented to slow the spread of the disease. This portion of the population have a different rate of infection. At the same time, there may be transfers from non-compliant ($u_4,u_5,u_6$) to compliant ($u_1,u_2,u_3$) behaviour and vice versa, for instance, when individuals get vaccinated or infected, or the individual decides not to wear a mask, respectively. Further population divisions can also be considered, including different age groups, genders or blood groups, which may have different infectivity rates.

In this model, we introduce the nonlocal operator $\mathcal{L}$ to represent the nonlocal diffusive behaviour of humans, which can arise from various factors such as familial ties, job requirements, and access to amenities. Such nonlocal models help to deepen our understanding of complex behaviours, in comparison to classical models. Our main result for these models is then given by the following:
\begin{theorem}
Let $\mathbf{F}\in\mathcal{A}$ and $\mathbf{q}^n=(q_1^n,\dots,q_M^n)$ be such that $\{q_i^n\}_{n=1}^{\infty}$ is a complete set in $L^2(\Omega)$ for each $i=1,\dots,M$. Take $\mathbf{h}\in [C_c^{2,1}(\Omega_a\times(0,T))]^M$ to be given nonzero nonnegative functions. Assume that $p_i(x)\in C(\overline\Omega)$, $p_i\leq 0$ on $\Omega$ for every $i$. 
Let $\mathbf{u}^{n}(x,t;\mathbf{p}^1,\mathbf{F}^1)$, $\mathbf{u}^{n}(x,t;\mathbf{p}^2,\mathbf{F}^2)$  be the bounded classical solutions of problem \eqref{SpaceFracMainPb} corresponding to the potentials $\mathbf{p}^k$ and interaction functions $\mathbf{F}^k$ ($k=1,2$) respectively. If 
\begin{equation}
    \langle\Lambda^1_{\mathbf{p}^1,\mathbf{F}^1},\mathbf{h}\rangle=\langle\Lambda^1_{\mathbf{p}^2,\mathbf{F}^2},\mathbf{h}\rangle
\end{equation}
then 
\[\mathbf{p}^1=\mathbf{p}^2\quad\text{ in }\Omega\quad \text{ and }\quad \mathbf{F}^1=\mathbf{F}^2.\]
\end{theorem}

Here, $\mathbf{q}$ may depict various environmental factors, such as fluctuations in the number of organisms carrying pathogens, sudden changes in pathogen numbers due to weather variations, and natural births and deaths within the population. This aforementioned theorem then states that with the introduction of these environmental factors into the model, we can uniquely determine the infection, recovery and death rates of the population. This result holds significant value in the study of disease dynamics, population regulation, and the spread of infectious diseases. Moreover, it can provide valuable insights for the development of effective control strategies to mitigate the impact of such diseases. 

\subsection{Fractional Viscoelastic Models of Cells and Tissues Based on the Nonlocal Nonlinear Parabolic System}
Cells, as the fundamental units of life, undergo a myriad of mechanical processes to fulfill their biological functions. These processes often involve intricate changes in cell shape, adhesion to surfaces, and even cell division. To accomplish these tasks, cells adapt and respond to local stresses by deforming and restructuring their internal components, including the cytoplasm, cytoskeleton, and cell membrane. The complex nature of cells as gel-like entities dispersed within fluid-filled compartments interconnected by elastic support elements gives rise to a distributed strain response within the cell.

When cells experience compression, torsion, or shear stresses, their response goes beyond simple deformation. It entails a coordinated interplay of mechanical forces that induce both structural changes and fluid flow within the cell, known as cellular microrheology. To explore and understand these intricate mechanical properties at the microscopic scale, bioengineers have harnessed multiple cutting-edge technologies such as microelectrochemical systems \cite{CellMechanics1} and nanotechnologies including atomic force microscopy \cite{CellMechanics2}, magnetic cytometry \cite{CellMechanics3} and optical tweezers \cite{CellMechanics4}.

Viscoelastic models have emerged as valuable tools for characterising the mechanical responses exhibited by cells under various stimuli. In recent years, fractional order constitutive equations have been incorporated into these models to capture the intricate viscoelastic properties of cells \cite{CellMechanics3,CellMechanics5}. By employing these advanced models, researchers can investigate phenomena such as stress relaxation, where cells dissipate internal stresses over time, and creep compliance, which describes the progressive deformation of cells under a constant load. These models offer flexibility and versatility, enabling the description of cellular responses across a wide range of stimulus conditions, thereby providing a novel conceptual framework for cellular microrheology.

The rheological behaviour of a variety of different cells can be modelled by the following nonlocal differential equation (see, for instance, \cite[Chapter X.C.1]{FracBio} or \cite{JiangWang2024FracCell})
\begin{equation}\label{ViscoelasticCellModel}
\partial_t \textbf{u}-\mathcal{L} \textbf{u}=-\textbf{G}_S(x)\textbf{u}+\textbf{F}(x,t,\textbf{u})+\textbf{q} \quad\text{ in }\  \Omega\times(0,T],
\end{equation}
where $u_i$ represents the strain of each type of cell, $\mathbf{F}$ denotes cell-cell interactions, $\mathbf{G}_S$ is the static elastic modulus which is different for each type of cell, $\mathbf{q}$ signifies external stress applied to the cells, and the nonlocal operator $\mathcal{L}$ reflects the nonlocal relaxation of the cell due to its complex composition with varying viscosities.

Then our main results state that
\begin{theorem}
Let $\mathbf{F}\in\mathcal{A}$ and $\mathbf{q}^n=(q_1^n,\dots,q_M^n)$ be such that $\{q_i^n\}_{n=1}^{\infty}$ is a complete set in $L^2(\Omega)$ for each $i=1,\dots,M$. Take $\mathbf{h}\in [C_c^{2,1}(\Omega_a\times(0,T))]^M$ to be given nonzero nonnegative functions. Assume that $\mathbf{G}_S(x)\in [C(\overline\Omega)]^M$, $\mathbf{G}_S\geq \mathbf{0}$ on $\Omega$. 
Let $\mathbf{u}^{n}(x,t;\mathbf{G}_S^1,\mathbf{F}^1)$, $\mathbf{u}^{n}(x,t;\mathbf{G}_S^2,\mathbf{F}^2)$  be the bounded classical solutions of problem \eqref{ViscoelasticCellModel} corresponding to the elastic modulus $\mathbf{G}_S^k$ and cell-cell interaction functions $\mathbf{F}^k$ ($k=1,2$) respectively. If 
\begin{equation}
    \langle\Lambda^1_{\mathbf{G}_S^1,\mathbf{F}^1},\mathbf{h}\rangle=\langle\Lambda^1_{\mathbf{G}_S^2,\mathbf{F}^2},\mathbf{h}\rangle
\end{equation}
then 
\[\mathbf{G}_S^1=\mathbf{G}_S^2\quad\text{ in }\Omega\quad \text{ and }\quad \mathbf{F}^1=\mathbf{F}^2.\]
\end{theorem}

This means that by applying a variety of stresses $\mathbf{q}$, we can uniquely determine the elastic modulus $G_S$ for all the different types of cells, as well as their analytic nonlinear coupled interactions $\mathbf{F}$. Hence, by exploring stress relaxation, creep compliance, and other mechanical phenomena, these viscoelastic models provide valuable insights into the intrinsic characteristics of cells. They help us understand how cells perceive and respond to mechanical cues within themselves and with their microenvironment. By unraveling the intricacies of cellular microrheology, these studies contribute to our understanding of fundamental biological processes, such as embryonic development, tissue morphogenesis, wound healing, and disease progression. Ultimately, this expanded knowledge paves the way for the development of innovative approaches and therapeutic strategies that leverage the mechanical properties of cells for regenerative medicine, tissue engineering, and drug delivery applications.

\subsection{Brain Tumour Growth Model Based on the Time-Fractional Diffusion System}
Brain tumours are abnormal cell growths that can occur in the brain or spinal cord, affecting various functions of the nervous system. Typically, imaging tests such as magnetic resonance imaging (MRI), advanced brain tumour imaging (ABTI), computed tomography (CT) scan and positron emission tomography (PET) scan are first conducted to detect the location and characteristics of the tumour. Yet, current imaging tests are unable to differentiate between various types and grades of brain tumours. Often, it is necessary for a biopsy to be conducted to obtain a tissue sample for further analysis. However, this procedure carries many risks, including bleeding and infection. Therefore, computational methods and mathematical models \cite{TumourClassical1,TumourClassical2,TumourClassical3} are valuable supplements to traditional diagnostic methods, by estimating tumour parameters, simulating tumour growth, and predicting tumour behaviour under different conditions, to evaluate treatment effectiveness and aid in treatment planning. In recent years, fractional calculus has been introduced into such models \cite{TumourFrac1,TumourFrac2,TumourFrac3}, and has been found to offer a more accurate representation of the physical scenario. This is due to increased degrees of freedom and nonlocal memory properties exhibited by fractional derivatives, in comparison with classical derivatives, while still weighing more strongly on current history than distant past events. Furthermore, chemotherapy drug diffusion and reaction can also be incorporated into the model, with the introduction of fractional derivatives.

This model is given by the nonlinear time-fractional proliferation-invasion model:
\begin{equation}\label{TumourModel}
\sum_{j=1}^{B}b_j({_{0}D}_t^{\beta_j}\mathbf{u})-\mathbf{d}\Delta \mathbf{u}=\bm{\rho} \mathbf{u}(1-\bm{\kappa}\mathbf{u}),\quad\text{ in }\Omega\times(0, T],
\end{equation}
where $u_i$ signifies the different types of cells (e.g. tumour cells, immune cells such as lymphocytic T-cells and natural killer cells, drug cells), the anisotropic diffusion coefficient $d_i$ of these cells depends on the heterogeneous media (e.g. cerebrospinal fluid, gray matter of brain, white matter of brain, skull, etc. for brain tumours), $\rho_i$ is the proliferation rate indicating the ability of the cells to divide and multiply, based on a multiplication factor $\kappa_i\in\mathbb{R}$. The fractional time derivative describes the long-tailed nonlocal anisotropic dependence of the current tumour growth rate on its previous growth, location as well as the response to the immune cells and chemotherapy drugs.

In order to better understand the characteristics of the tumours, it is a common practice to administer contrast dye during imaging procedures \cite{ContrastImaging1,ContrastImaging2,SiniContrastImaging}. The presence of this contrast dye may be incorporated into the model \eqref{TumourModel} by introducing an additional source function $\mathbf{q}$ that represents its manifestation:
\begin{equation}\label{TumourModelQ}
{_{0}D}_t^{\bm{\beta}}\mathbf{u}-\mathbf{d}\Delta \mathbf{u}=\bm{\rho} \mathbf{u}(1-\bm{\kappa}\mathbf{u})+\mathbf{q},\quad\text{ in }\Omega\times(0, T].
\end{equation}
Then, our main result in Theorem \ref{TimeFracThm} is given by 
\begin{theorem}
Let $\mathbf{q}^n=(q_1^n,\dots,q_M^n)$ be such that $\{q_i^n\}_{n=1}^{\infty}$ is a complete set in $L^2(\Omega\times(0,T))$ for each $i=1,\dots,M$. For $\Gamma\subseteq\partial\Omega$, take $\mathbf{h}\in [C_c(\Gamma\times(0,T))]^M$ to be given nonzero nonnegative functions. Assume that $\rho_i(x,t)\in C(\overline\Omega\times[0,T])$, $\rho_i\leq 0$ on $\Omega\times(0,T]$ for every $i$. 
Let $\mathbf{u}^{n}(x,t;\bm{\rho}^1,\bm{\kappa}^1,\bm{\beta}^1)$, $\mathbf{u}^{n}(x,t;\bm{\rho}^2,\bm{\kappa}^2,\bm{\beta}^2)$  be the bounded classical solutions of problem \eqref{TumourModelQ} corresponding to $\bm{\rho}^k$ and $\bm{\kappa}^k$ ($k=1,2$) respectively. 
If 
\begin{equation}
    \Lambda^2_{\bm{\rho}^1,\bm{\kappa}^1}=\Lambda^2_{\bm{\rho}^2,\bm{\kappa}^2}\quad \text{ and }\quad \Lambda^3_{\bm{\beta}^1}=\Lambda^3_{\bm{\beta}^2}
\end{equation}
then 
\begin{equation}\bm{\rho}^1=\bm{\rho}^2\quad\text{ in }\Omega\times(0,T) \quad\text{ and }\quad\bm{\kappa}^1=\bm{\kappa}^2\quad \text{ and }\quad \bm{\beta}^1=\bm{\beta}^2.\end{equation}
\end{theorem}

This result means that the introduction of a dense set of contrast dyes has multiple physical implications in the model as well as for the imaging technique. It enables the unique determination of key parameters for all the different kinds of cells, including the proliferation rates $\rho_i$, the multiplication factors $\kappa_i$, as well as the nonlocal time-anisotropy $\bm{\beta}$, based on partial averaged-out flux measurements $\Lambda^2$ along part of the boundary $\Gamma\subset\partial\Omega$, as well as measurements $\Lambda^3$ at a single interior point $\{x_0\}\times(0,T)$. 

Therefore, with the introduction of contrast dyes $\mathbf{q}$, the model becomes capable of extracting valuable information regarding the growth and temporal behaviour of the tumour. The understanding of these parameters contributes to a more comprehensive understanding of the tumour's behaviour and characteristics, facilitating the formulation of more effective treatment strategies.

\subsection{Biological Pattern Formation Based on Time-Fractional Reaction-Diffusion Systems}

Reaction-diffusion systems play a vital role  in describing and understanding the spatial patterns that emerge through chemical processes within cells, including mazes, stripes and spots, that are present in hydra pattern formation, shell pigmentation, and animal coat markings \cite{PatternFormAnimal1,PatternFormAnimal2,TuringPatterns}. By introducing the Caputo time-fractional operator, the model is now capable of describing more complex patterns. In particular, it is known that the time-fractional reaction-diffusion system can now give rise to travelling wave patterns \cite{FracPatternFormAnimal}, instead of the transient, oscillating or stationary structures associated to the classical system. With the inclusion of nonlocality and memory effects through the fractional time derivative, the model can now incorporate the influence of previous substance concentrations and reaction states into present dynamics. This enhancement leads to the emergence of fascinating new patterns including spiral and irregular patterns, as well as diamond, cyclic, and star-like structures \cite{ghafoor2024dynamics}. These biological pattern formations are mathematically modelled by the following time-fractional system of equations
\begin{equation}\label{ReactionDiffusion}
{_{0}D}_t^{\bm{\beta}}\mathbf{u}-\mathbf{d}\Delta \mathbf{u}=\mathbf{F}(\mathbf{u}),\quad\text{ in }\Omega\times(0, T].
\end{equation}
Two examples of reaction-diffusion systems are the Gray-Scott model and the Schnakenberg model. 

% The Brusselator model is a mathematical model that captures essential dynamics of self-organisation and pattern formation in biological systems. Based upon autocatalytic chemical reactions, the Brusselator model is able to describe a diverse range of pattern formation in biological systems, including the development of intricate patterns in animal coat markings, pigmentation patterns in shells, and tissue patterning during embryonic development. The Brusselator model consists of several coupled nonlinear equations representing the concentrations of multiple chemical species: activators (A) and inhibitors (I). The model incorporates autocatalytic reactions, where the activators stimulate their own production and the inhibitors counteract the effects of the activators. The interactions between the activators and inhibitors give rise to the dynamics that lead to pattern formation. This time-fractional Brusselator model is given by the following system of equations:
% \begin{equation}\label{Brusselator}
% \begin{split}
%     \sum_{j=1}^{B}b_j({_{0}D}_t^{\beta_j}\mathbf{u})-\mathbf{d}\Delta \mathbf{u}=(-c-1)\mathbf{u}+\mathbf{u}^2\mathbf{v}+\mathbf{q}_1(x,t),\quad\text{ in }\Omega\times(0, T],\\
% \sum_{j=1}^{B}b_j'({_{0}D}_t^{\beta_j'}\mathbf{v})-\mathbf{d}'\Delta \mathbf{v}=c\mathbf{u}-\mathbf{u}^2\mathbf{v}+\mathbf{q}_2(x,t),\quad\text{ in }\Omega\times(0, T].
% \end{split}
% \end{equation}

The Gray-Scott model is a mathematical model that has been widely used to study spatial patterns and self-organisation in biological systems. It has been applied to understand the formation of self-reproducing spot-like patterns and the dynamics of pattern replication. It captures the interplay between multiple chemical species, typically referred to as activators and inhibitors. The activators promote their own production and the production of the inhibitors, while the inhibitors suppress the production of both themselves and the activators. This autocatalytic and inhibitory interaction gives rise to the complex dynamics observed in biological pattern formation. These patterns often exhibit a characteristic spot-like morphology, where regions of high activator concentrations are surrounded by regions of low activator concentrations. Such patterns can be seen in a variety of animals, including the stripes on a zebra, the spots on a cheetah, as well as the pigment patterns on fish skin, butterfly wings and flower petals. Hence, the Gray-Scott model provides insights into the underlying mechanisms governing pattern formation and describes how genetic and environmental factors influence these patterns. The time-fractional Gray-Scott model with external environmental source functions $\mathbf{q}_1,\mathbf{q}_2$ is given by the following system \cite{ghafoor2024dynamics}:  
\begin{equation}\label{GrayScott}
\begin{split}
    {_{0}D}_t^{\bm{\beta}}\mathbf{u}-\mathbf{d}\Delta \mathbf{u}=-\gamma\mathbf{u}-m\mathbf{u}\mathbf{v}^2+\mathbf{q}_1(x,t),\quad\text{ in }\Omega\times(0, T],\\
{_{0}D}_t^{\bm{\beta}'}\mathbf{v}-\mathbf{d}'\Delta \mathbf{v}=(-\gamma-c)\mathbf{v}+m\mathbf{u}\mathbf{v}^2+\mathbf{q}_2(x,t),\quad\text{ in }\Omega\times(0, T].
\end{split}
\end{equation}
Here, $\mathbf{d}$ and $\mathbf{d}'$ describe the diffusion rates of the activators $\mathbf{u}$ and inhibitors $\mathbf{v}$, $m$ represents the reaction rate, $\gamma$ represents the input rate of the activators $\mathbf{u}$, while $c$ represents the removal or ``kill" rate.

Then, our main result in Theorem \ref{TimeFracThm} is given by 
\begin{theorem}
Let $\{\mathbf{q}_1^n\}_{n=1}^{\infty},\{\mathbf{q}_2^n\}_{n=1}^{\infty}$ be complete sets in $L^2(\Omega\times(0,T))$. For $\Gamma\subseteq\partial\Omega$, take $\mathbf{h}\in [C_c(\Gamma\times(0,T))]^M$ to be given nonzero nonnegative functions. For $\gamma,c>0$, let $(\mathbf{u}^{n}_k,\mathbf{v}^{n}_k)$ be the bounded classical solutions of problem \eqref{GrayScott} corresponding to $\gamma^k$, $c^k$, $m^k$ and $\bm{\beta}^k$ ($k=1,2$) respectively. 
If 
\begin{equation}
    \Lambda^2_{\gamma^1,c^1,m^1}=\Lambda^2_{\gamma^2,c^2,m^2}\quad \text{ and }\quad \Lambda^3_{\bm{\beta}^1}=\Lambda^3_{\bm{\beta}^2}
\end{equation}
then 
\begin{equation}\gamma^1=\gamma^2,\quad c^1=c^2,\quad m^1=m^2, \quad \text{ and }\quad \bm{\beta}^1=\bm{\beta}^2.\end{equation}
\end{theorem}

This theorem implies that by introducing a dense set of external environmental factors $\mathbf{q}_1,\mathbf{q}_2$, we can uniquely determine the various rates of reactions, including the input rate, reaction rate, and consumption rate. 

Another biological model describing activation-inhibition is the Schnakenberg model, also known as the activator-depleted model. Proposed by Julius Schnakenberg in 1979 \cite{SCHNAKENBERG1979389}, the Schnakenberg model shares similarities with the Gray-Scott model but allows for autocatalysis, where activators stimulate their own production, unlike the Gray-Scott model that requires feeding of activators. This forms a positive feedback loop for activators, which is counteracted by the negative inhibition feedback action of inhibitors. The time-fractional Schnakenberg model has been shown to display a variety of complex spatial patterns such as stripes, spots, or labyrinthine patterns \cite{ghafoor2024dynamics}, and is given as follows:
\begin{equation}\label{Schnakenberg}
\begin{split}
    {_{0}D}_t^{\bm{\beta}}\mathbf{u}-\mathbf{d}\Delta \mathbf{u}=-c\mathbf{u}+\gamma\mathbf{u}^2\mathbf{v}+\mathbf{q}_1(x,t),\quad\text{ in }\Omega\times(0, T],\\
{_{0}D}_t^{\bm{\beta}'}\mathbf{v}-\mathbf{d}'\Delta \mathbf{v}=-\gamma\mathbf{u}^2\mathbf{v}+\mathbf{q}_2(x,t),\quad\text{ in }\Omega\times(0, T].
\end{split}
\end{equation}

In this case, our theorem is stated as follows:
\begin{theorem}
Let $\{\mathbf{q}_1^n\}_{n=1}^{\infty},\{\mathbf{q}_2^n\}_{n=1}^{\infty}$ be complete sets in $L^2(\Omega\times(0,T))$. For $\Gamma\subseteq\partial\Omega$, take $\mathbf{h}\in [C_c(\Gamma\times(0,T))]^M$ to be given nonzero nonnegative functions. For $\gamma,c>0$, let $(\mathbf{u}^{n}_k,\mathbf{v}^{n}_k)$ be the bounded classical solutions of problem \eqref{Schnakenberg} corresponding to $\gamma^k$, $c^k$ and $\bm{\beta}^k$ ($k=1,2$) respectively. 
If 
\begin{equation}
    \Lambda^2_{\gamma^1,c^1}=\Lambda^2_{\gamma^2,c^2}\quad \text{ and }\quad \Lambda^3_{\bm{\beta}^1}=\Lambda^3_{\bm{\beta}^2}
\end{equation}
then 
\begin{equation}\gamma^1=\gamma^2,\quad c^1=c^2,\quad \text{ and }\quad \bm{\beta}^1=\bm{\beta}^2.\end{equation}
\end{theorem}

Physically, this result states that we are able to uniquely determine the rates of autocatalysis rates and activation-inhibition reactions in the system, by controlling the concentration of the external reactants represented by $\mathbf{q}_1$ and $\mathbf{q}_2$. By understanding these reaction parameters, we gain deeper insights into the mechanisms underlying the emergence of spatial patterns and self-organisation in biological processes.

\vspace{2em}

\noindent\textbf{Acknowledgment.} 
	The work was supported by the Hong Kong RGC General Research Funds (No. 12302919, 12301420 and 11300821), the NSFC/RGC Joint Research Fund (No. N\_CityU101/21), and the ANR/RGC Joint Research Grant (No. A\_CityU203/19).

\noindent\textbf{Conflict of interest statement.} 
	The authors declare that they have no conflict of interests that could have appeared to influence the work reported in this paper.

\noindent\textbf{Data availability statement.} 
We do not analyse or generate any datasets, because our work proceeds within a theoretical and mathematical approach.

\bibliographystyle{plain}
\bibliography{ref,refInversePb,refMFG}
%\printbibliography

\end{document}